\documentclass[11pt, a4paper]{amsart}
\usepackage{amsmath}
\usepackage{amssymb}
\usepackage{latexsym}
\usepackage[margin=2cm]{geometry}
\usepackage{epsfig}
\usepackage{epstopdf}
\usepackage{caption}

\usepackage{color}

\usepackage{etoolbox}
\patchcmd{\section}{\scshape}{\bfseries}{}{}
\patchcmd{\subsubsection}{\itshape}{\bfseries}{}{}
\makeatletter
\renewcommand{\@secnumfont}{\bfseries}
\makeatother

\newcommand{\Z}{\ensuremath{\mathbb Z}}
\newcommand{\R}{\ensuremath{\mathbb R}}
\newcommand{\HH}{\ensuremath{\mathbb H}}

\newcommand{\X}{\ensuremath{\mathbb X}}

\newcommand{\T}{\ensuremath{{\mathcal T}}}

\renewcommand{\rho}{\varrho}
\renewcommand{\phi}{\varphi}

\newtheorem{thm}{Theorem}[section]
\newtheorem{defi}[thm]{Definition}

\newtheorem{lem}[thm]{Lemma}

\theoremstyle{remark}

\newtheorem{rem}[thm]{Remark}

\begin{document}

\title{Symmetries of Monocoronal Tilings}

\author{Dirk Frettl\"oh}
\address{Technische Fakult\"at, Univ.~Bielefeld}
\email{dirk.frettloeh@udo.edu}
\author{Alexey Garber}
\address{Department of Geometry and Topology, Faculty of Mechanics and Mathematics,
Moscow State University}
\email{alexeygarber@gmail.com}
\begin{abstract}
The vertex corona of a vertex of some tiling is the vertex together with the 
adjacent tiles. A tiling where all vertex coronae 
are congruent is called monocoronal. We provide a classification of monocoronal
tilings in the Euclidean plane and derive a list of all possible symmetry 
groups of monocoronal tilings. In particular, any monocoronal tiling with
respect to direct congruence is crystallographic, whereas any monocoronal
tiling with respect to congruence (reflections allowed) is either
crystallographic or it has a one-dimensional translation group. 
Furthermore, bounds on the number of the dimensions of the translation 
group of monocoronal tilings in higher dimensional Euclidean space are obtained.
\end{abstract}

\maketitle

\section{Introduction}

The question how local properties determine global order of spatial structures
arises naturally in many different contexts. If the structures under
consideration are crystallographic structures, for instance, 
tilings are used frequently as models for these structures. In this context, 
one studies how local properties of a tiling determine its global properties. 

To mention just a few results in this context for tilings: A tiling is 
called {\em monohedral} if all tiles are congruent. A tiling is called
{\em isohedral} if its symmetry group acts transitively on the tiles.
Hilbert's 18th problem asked whether there is a tile admitting a monohedral 
tiling, but no isohedral tiling. (In fact, Hilbert asked the question
for tilings in three dimensions.) Heesch gave an affirmative answer in 
1935 by proposing a planar non-convex tile with this property
\cite{heesch}, \cite{gs}. A classification of all isohedral tilings by 
convex polygons was obtained by Reinhardt \cite{reinh}, see also 
Section 9 of the wonderful book by Gr\"unbaum and Shephard \cite{gs}. 
In contrast, a classification of all monohedral 
tilings by convex polygons has not been obtained yet. For instance,
it is still unknown which types of convex pentagons admit monohedral
tilings of the Euclidean plane \cite[Section 9]{gs}.

In a similar manner, a {\em monogonal} tiling is one in which each
vertex star --- i.e., a vertex together with its incident edges ---
is congruent to any other vertex star. A tiling is called {\em isogonal}, 
if its symmetry group acts transitively on the vertex stars. A 
classification of isogonal tilings is contained in \cite[Section 6]{gs}.
There are 11 combinatorial types of isogonal tilings and 91 types
of (unmarked, normal) isogonal tilings altogether.
Each of the 11 combinatorial types of isogonal tilings can be realised
by an {\em Archimedean} (aka {\em uniform}) tiling, that is a vertex
transitive tiling by regular polygons (which then is automatically isogonal)
\cite[Section 2]{gs}.

The Local Theorem for tilings yields a necessary and sufficient criterion for
a tiling to be crystallographic in terms of the number and the symmetries
of $k$-th tile coronae in the tiling \cite{dol-sch1}. (The $0$th tile corona 
of a tile $T$ is $T$ itself. The $k$-th tile corona of $T$ is the set of 
all tiles sharing a facet with some tile of the $k-1$-th tile corona of $T$.) 
As a consequence of this result, a monohedral tiling is isohedral if and only 
if all first tile coronae in $\T$ are congruent \cite{dol-sch2}. 

In light of the facts above it seems natural to explore tilings in which all 
vertex coronae are congruent. We will call these tilings {\em monocoronal}.
The {\em vertex corona} of a vertex $x$ in a tiling is $x$ together with the
tiles incident to $x$. Thus, monocoronal is a stronger property than
monogonal. 

The main results of this paper are obtained for tilings in the Euclidean 
plane $\R^2$ (Section \ref{sec:plane}). It turns out that 
a tiling in $\R^2$ is necessarily isogonal (hence crystallographic) if
all vertex corona are directly congruent (i.e., mirror images forbidden). 
This is stated in Theorem \ref{thm:e2-dir}.
In contrast, if all vertex corona of some tiling $\T$ are congruent, but
not necessarily directly congruent (i.e., mirror images allowed), then 
the translation subgroup of the symmetry group of $\T$ can 
be one-dimensional. In particular, $\T$ may be neither crystallographic 
nor isogonal. This is Theorem \ref{thm:e2-cong}. These results are obtained 
using a complete classification of monocoronal tilings.
This classification is contained in Appendices  \ref{class-facetoface} 
and \ref{class-nonfacetoface}. 

In Section \ref{sec:highdim} the results of Section \ref{sec:plane} are
used to obtain some monocoronal tilings in higher dimensional 
Euclidean space with small translation groups. Section \ref{sec:concl}
briefly illustrates the situation in hyperbolic spaces and states some
suggestions for further work. 

\subsection{Definitions and Notations}
Let $\X$ be a Euclidean space $\R^d$ or a hyperbolic space $\HH^d$. 
A {\em tiling} is a countable collection $\{T_1, T_2, \ldots\}$ of compact 
sets $T_i$ (the {\em tiles}) that is a covering (i.e., $\bigcup_i T_i = \X$) 
as well as a packing (i.e. $\mathring{T_i} \cap \mathring{T_j}=\varnothing$ 
if $i \ne j$, where $\mathring{T}$ denotes the interior of $T$). Here the tiles 
will almost always be convex polytopes; if not, it is mentioned explicitly.
A {\em vertex} of a tiling $\T$ is a point $x$ such that $x$ is a vertex of 
at least one tile in $\T$. Note that in general (higher dimensions, 
non-convex tiles) a proper definition of a vertex of a tiling can be 
problematic, compare \cite{gs} or \cite{FG}. Our definition is tailored to
monocoronal tilings. It agrees with the usual definitions and properties
of a vertex of a tiling if one considers only planar tilings by convex 
polygons. For instance, $x$ is a vertex of a tiling $\T$ if and only if $x$ is 
an isolated point of the intersection of some tiles in $\T$; or: every
vertex of a tile (in the usual sense: vertex of a polygon) is a vertex of the
tiling, compare \cite{FG}.

A tiling is called {\em face-to-face} if the intersection of two tiles is 
always an entire face (possibly of dimension less
than $d$, possibly empty) of both of the tiles. In particular, planar 
tilings by convex polygons are face-to-face if the intersection of two tiles is either an
entire edge of both of the tiles, or a vertex of both of the tiles, or empty.

\begin{defi}
Let $x$ be a vertex in some tiling $\T$. The {\em vertex-corona} of $x$
is the set of all tiles $T \in \T$ such that $x \in T$, together with $x$.
\end{defi}

In tilings by non-convex tiles, two different vertices may have the same set
of adjacent tiles. Thus it is necessary to keep track of the defining vertex $x$ 
in the definition of the vertex corona.

\begin{defi}
If all vertex-coronae in a tiling $\T$ are directly congruent (mirror images
forbidden), then we say that $\T$ is a {\em monocoronal tiling up
to rigid motions}. If all vertex-coronae in a tiling $\T$ are congruent 
(mirror images allowed), then we say that $\T$ is a {\em monocoronal 
tiling up to congruence}. If it is clear from the context which
of the both terms is meant then we will say briefly that $\T$ is a
{\em monocoronal tiling}.  
\end{defi}

The {\em symmetry group} of a tiling $\T$ is the set of all isometries
$\varphi: \X \to \X$ such that $\varphi(\T)=\T$. 
In the sequel we are interested in possible symmetry groups of
monocoronal tilings. In particular we ask whether only
crystallographic groups can occur.

\begin{defi} 
A tiling in $\R^d$ is called {\em $k$-periodic} if the symmetry group contains
exactly $k$ linearly independent translations. A $0$-periodic tiling is 
called {\em non-periodic}.
\end{defi}

\begin{defi}
The tiling $\mathcal{T}$ is called {\em crystallographic} if its symmetry group 
has compact fundamental domain. Otherwise it is called {\em non-crystallographic}.
\end{defi}

In Euclidean space $\mathbb{R}^d$ the two notions ``crystallographic''
and ``$d$-periodic" are equivalent.

\begin{thm}[Bieberbach, 1911-1912, \cite{Bieb1, Bieb2}]
A tiling $\mathcal{T}$ of $\mathbb{R}^d$ is crystallographic iff it is 
$d$-periodic.
\end{thm}

Since in hyperbolic spaces there is no natural meaning of ``translation''
it does not make sense to speak of ``periodic'' or ``non-periodic''
tilings in $\HH^d$. Nevertheless we can ask whether a hyperbolic 
monocoronal tiling is necessarily crystallographic. This is answered
in Section \ref{sec:concl}.

We will use orbifold notation to denote planar symmetry groups in the sequel, 
compare \cite{BCG}. For instance, $\ast442$ denotes the symmetry group of the 
canonical face-to-face tiling of $\R^2$ by unit squares. For a translation of 
orbifold notation into other notations see \cite{BCG} or \cite{wik-orb}.
In orbifold notation the 17 crystallographic groups in the Euclidean plane 
(``wallpaper groups'') are
\[ \ast632, \, \ast442, \, \ast333, \, \ast2222, \, \ast \ast, \, \ast
\times, \, \times \times, \quad 632, \, 442, \, 333, \, 2222, \, \circ,
\quad 4\ast 2, \, 3 \ast 3, \, 2 \ast 22, \, 22 \ast, 22 \times; \]
and the seven frieze groups are
\[ \infty \infty, \, \ast \infty \infty, \, \infty \ast, \, \infty \times,
\, 22 \infty, \, 2 \ast \infty, \, \ast 22 \infty. \]

\section{Euclidean plane}\label{sec:plane}

This section is dedicated to the question: What are possible symmetry 
groups of monocoronal tilings in $\R^2$?
It turns out that the answer is rather different depending on whether we 
consider monocoronal tiling up to rigid motions, or 
up to congruence. In the sequel the main results are stated first. It
follows a subsection containing an outline of the classification of all 
monocoronal tilings that are face to face, then a subsection sketching the
classification of monocoronal tilings that that are not face-to-face, 
and finally the proofs of Theorems \ref{thm:e2-dir} and \ref{thm:e2-cong}
are given.

\begin{thm} \label{thm:e2-dir}
Every monocoronal tiling up to rigid motions has one of the 
following 12 symmetry groups:
\[ \ast632, \, \ast442, \, \ast333, \, \ast2222, \, \quad 632, \, 442, \, 333, 
\, 2222, \quad 4\ast2, \, 3 \ast 3, \, 2 \ast 22, \, 22 \ast. \]
In particular, every such tiling is crystallographic, every such tiling 
has a center of rotational symmetry of order at least 2, and its symmetry
group acts transitively on the vertices.
\end{thm}

\begin{thm} \label{thm:e2-cong}
Every monocoronal tiling up to congruence is either 
1-periodic, or its symmetry is one out of 16 wallpaper groups:
any except $\ast \times$. If such tiling is 1-periodic then its symmetry 
group is one of four frieze groups: $\infty\infty$, $\infty \times$, 
$\infty \ast$, or $22\infty$.
In particular, every such tiling is crystallographic or 1-periodic.
\end{thm}

\begin{rem}
It turns out that the 1-periodic tilings in Theorem \ref{thm:e2-cong}
consist of 1-periodic layers that are stacked according to a non-periodic
one-dimensional sequence (see for instance Figure \ref{pict:3dim-block},
page \pageref{pict:3dim-block}).
Such a sequence may look simply like 
$\ldots 1,1,1,0,1,1,1, \ldots$ (all 1s, one 0) or, more 
interestingly, like a non-periodic Fibonacci
sequence $\ldots0,1,1,0,1,0,1,1,0,1,1,0,1,0 \ldots$, where each finite
sub-sequence occurs infinitely often with bounded gaps. The study
of nonperiodic symbolic sequences is an interesting field of study
on its own, compare for instance \cite{fogg} or \cite{baa-grim}.
\end{rem}

\subsection{Face-to-face tilings}\label{subsec:ftf}

We start with determining the possible combinatorial types of vertex coronae 
in a monocoronal tiling; that is, in which way can every vertex 
be surrounded by $n$-gons. This first part is already contained in
\cite{gs}: Our Lemma \ref{lem:combtype} 
corresponds to Equation (3.5.5) in \cite{gs}, our Table \ref{tab:list}
corresponds to Table 2.1.1 in \cite{gs} (which is also applicable to 
the more general case considered here). For the sake of completeness
we include the line of reasoning in the sequel.

\begin{lem} \label{lem:combtype}
Let $\T$ be a monocoronal tiling such that every vertex is 
incident to $n$ polygons with $a_1,\ldots, a_n$ many edges, respectively. Then 
\[ \sum_{i=1}^n \frac{1}{a_i}=\frac{n}{2}-1.\]
\end{lem}

\begin{proof}
Let $k_i$ denote the number of $i$-gons in the corona of a vertex in a
monocoronal tiling $\mathcal{T}$. 
Consider the average sum of angles at every vertex. On one hand, it 
is equal to $2\pi$ since $\mathcal{T}$ is a tiling. On the other hand, every 
$i$-gon adds an angle $\frac{i-2}{i}\pi$ in average. By summing up 
contributions of all polygons and dividing by $\pi$ we obtain 
${\displaystyle \sum\limits_{i\geq 3} \frac{k_i(i-2)}{i}=2}$.
Let $n$ denote the total number of polygons incident to a vertex $x$ in 
$\mathcal{T}$ (i.e. $n=\sum k_i$) then this identity becomes
\[\sum\limits_{i\geq 3}\frac{k_i}{i}=\frac{n}{2}-1.\]
Reformulating this with respect to $a_i$ yields the claim.
\end{proof}

One can easily check that the equation in Lemma \ref{lem:combtype} is possible 
only in the cases listed in Table \ref{tab:list} (where only non-zero $k_i$'s are
listed).

\begin{table}
\begin{tabular}{|l|l|c|}
\hline
$n=6$ &  $k_3=6$ & \\
\hline
$n=5$ &  $k_3=4$, $k_6=1$ & \\
& $k_3=3$, $k_4=2$ & \\
\hline
$n=4$ &  $k_3=2$, $k_4=1$, $k_{12}=1$ & (i)\\
 & $k_3=2$, $k_6=2$ & \\
 & $k_3=1$, $k_4=2$, $k_6=1$ & \\
 & $k_4=4$ & \\
\hline
$n=3$ &  $k_3=1$, $k_7=1$, $k_{42}=1$ & (i)\\
 & $k_3=1$, $k_8=1$, $k_{24}=1$ & (i)\\
 & $k_3=1$, $k_9=1$, $k_{18}=1$ & (i)\\
 & $k_3=1$, $k_{10}=1$, $k_{15}=1$ & (i)\\
 & $k_3=1$, $k_{12}=2$ & (i)\\
 & $k_4=1$, $k_5=1$, $k_{20}=1$ & (i)\\
 & $k_4=1$, $k_6=1$, $k_{12}=1$ & \\
 & $k_4=1$, $k_8=2$ & \\
 & $k_5=2$, $k_{10}=1$ & (i)\\
 & $k_6=3$ & \\
\hline
\end{tabular}
\caption{Possible combinatorial types of vertex coronae following from
Lemma \ref{lem:combtype}. Cases that are impossible for further reasons
are marked with (i). \label{tab:list}}
\end{table}

Not all cases with $n=3$ and $n=4$ can be realized by a tiling in the
Euclidean plane. Impossible cases are marked by (i) in Table \ref{tab:list}. 
For example, there is no tiling in the Euclidean plane such that every vertex 
is incident to one triangle, one $7$-gon and one $42$-gon. 
Indeed, the edges of any triangle in such a tiling in cyclic
order would belong alternately to 7-gons and to 42-gons. Since a triangle
has an odd number of edges this is impossible. 

In the same way one can prove that if $n=3$ then only four cases can be 
realized: one triangle and two $12$-gons, or three hexagons, or one quadrilateral 
and two octagons, or one quadrilateral, hexagon and one $12$-gon. If $n=4$ then 
the case with two triangles, one quadrilateral and one $12$-gon cannot be realized.

The full list of all possible monocoronal face-to-face tilings presented in Appendix
\ref{class-facetoface}. In the sequel the general line of reasoning is illustrated
by a detailed discussion of a typical example, namely the case where the vertex 
corona consists of one triangle, two quadrilaterals, and one hexagon.

First we consider the topological structure of a tiling $\mathcal{T}$ 
with prescribed (combinatorial) polygons in the corona of every vertex. For 
example, if $n=6$ then the tiling $\mathcal{T}$ will be combinatorially 
equivalent to the tiling with regular triangles.

The example under consideration --- one triangle, two quadrilaterals, 
one hexagon --- is less trivial. In this case with  there are two possible
combinatorial coronas. They are shown in Figure \ref{pict:tr-2q-hex}.
\begin{figure}[!ht]
\begin{center}
\includegraphics[scale=2.5]{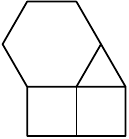}
\hskip 2cm 
\includegraphics[scale=2.5]{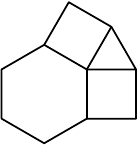}
\caption{Two theoretical cases of local structure with triangle, two quadrilaterals and hexagon.}\label{pict:tr-2q-hex}
\end{center}
\end{figure}
The constellation in the left part of the figure can not be realized in 
any tiling of the plane. The reason is the same as above: quadrilaterals 
and hexagons need to alternate in a cyclic way along the edges of some triangle,
yielding a contradiction.  
So the topological structure of the tiling that we are looking for will 
be the same as in Figure \ref{pict:tr-2q-hex:top}.
\begin{figure}[!ht]
\begin{center}
\includegraphics[scale=1]{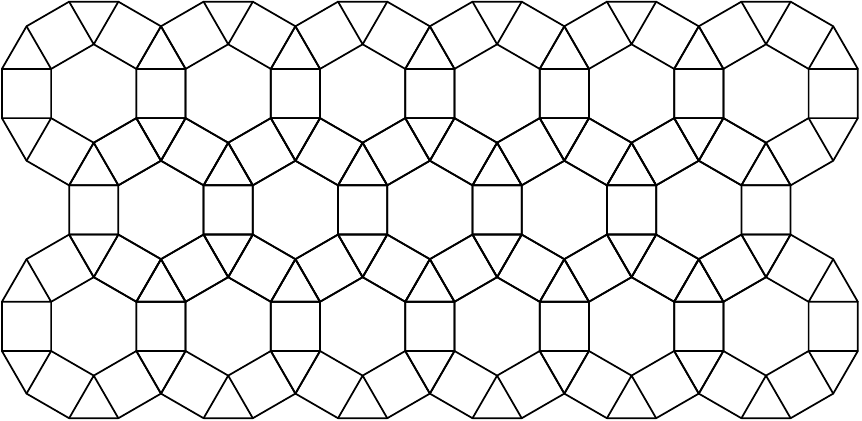}
\caption{Topological structure of a tiling with unique vertex corona consisting of 
a triangle, two quadrilaterals and a hexagon.}\label{pict:tr-2q-hex:top}
\end{center}
\end{figure}

The next step is to mark equal edges or predefined angles. (This is the
point where we go beyond \cite{gs}.) For example, in 
the unique topologically possible corona with triangle, two quadrilaterals and 
hexagon there is only one angle of triangle in every vertex, therefore all 
angles of all triangles in our tiling are equal to $\frac{\pi}3$. By the same 
reason all angles of all hexagons are equal to $\frac{2\pi}{3}$. 

We use the same strategy to mark edges of equal length. It is easy 
to see that all edges of all triangles are equal, since every corona 
contains only one equilateral triangle. Hence we mark all edges of all 
triangles with red color. The two edges of a hexagon that are incident to a given 
vertex $x$ may be of different lengths. Thus we mark them with blue and 
green, respectively. Then the edges of all hexagons are colored alternating 
with blue and green since every vertex must be incident to one blue and one 
green edge of a hexagon. Hence we have colored all edges in the corona except 
two of them (see Figure \ref{pict:tr-2q-hex:precolor}).
\begin{figure}[!ht]
\begin{center}
\includegraphics[scale=2.5]{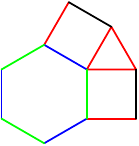}
\caption{Intermediate coloring of corona.}\label{pict:tr-2q-hex:precolor}
\end{center}
\end{figure}

Both uncolored edges belong to some hexagons so they can be blue or green.
First we will consider the case when green and blue are equal. Then both 
quadrilaterals in every corona are parallelograms. There are no more than 
two different angles in total in these quadrilaterals since in one corona 
there are only two angles of quadrilaterals. So metrically the corona of an 
arbitrary vertex $x$ looks as follows 
(see Figure \ref{pict:tr-2q-hex:nogreen}): a regular hexagon, two equal 
parallelograms that touch $x$ with different angles (or equal iff they 
are rectangles), and a regular triangle.
\begin{figure}[!ht]
\begin{center}
\includegraphics[scale=2.5]{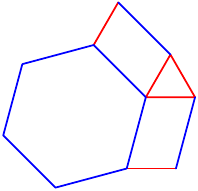}
\caption{First possible case of a corona with one triangle, two quadrilaterals, and one hexagon.}\label{pict:tr-2q-hex:nogreen}
\end{center}
\end{figure}

If the blue and the green edges are of different length then one of the 
uncolored edges in  Figure \ref{pict:tr-2q-hex:precolor} is green and the 
other one is blue. There are two arguments for that. The first one which works 
in this particular case is the following: if both edges are of the same 
color, say blue, then every corona does not have a green edge on the outer 
part of incident quadrilaterals, which is impossible. The second argument is 
general for all monocoronal tilings: in any corona with center 
$x$, the portion of green edges of quadrilaterals containing $x$ is the 
same as the portion of green edges of quadrilaterals not containing $x$. 
This means in this example: one quarter of all edges of quadrilaterals must 
be green.

So there are two possibilities of coloring all edges of the vertex corona 
under consideration. These are shown in Figure \ref{pict:tr-2q-hex:color}.
\begin{figure}[!ht]
\begin{center}
\includegraphics[scale=2.5]{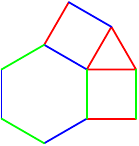}
\hskip 1cm
\includegraphics[scale=2.5]{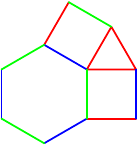}
\caption{Possible colorings of corona.}\label{pict:tr-2q-hex:color}
\end{center}
\end{figure}

In the left image both quadrilaterals are parallelograms. Moreover, every 
corona contains only one angle between red and blue edges, hence all 
these angles are equal. But four such angles are angles of one quadrilateral, 
so this quadrilateral is a rectangle. By the same reason the second 
quadrilateral is a rectangle, too. Now if we consider an arbitrary 
triangle then it is surrounded by rectangles that alternate cyclically 
along its edges: a rectangle with two red and two green edges is followed
by a rectangle with two red and  two blue edges. Again, since the
triangle has an odd number of edges, this yields a contradiction. 

By the same reason, in the right image all angles between red and blue 
edges are equal, and all angles between red and green edges are equal. 
So both quadrilaterals are equal isosceles trapezoids. Hence the metrical 
picture looks like shown in Figure \ref{pict:tr-2q-hex:metr}.

\begin{figure}[!ht]
\begin{center}
\includegraphics[scale=2.5]{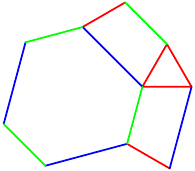}
\caption{The second possible corona with triangle, two quadrilaterals, and hexagon.}\label{pict:tr-2q-hex:metr}
\end{center}
\end{figure}

Putting everything together we obtain that there are two families of different 
monocoronal tilings whose vertex corona consists of one triangle, two quadrilaterals, 
and one hexagon. Both families are shown in Figure \ref{pict:tr-2q-hex:final}. 
Both families admit only crystallographic tilings since the metrical structure is 
uniquely defined by the structure of an arbitrary corona satisfying 
Figures \ref{pict:tr-2q-hex:nogreen} or \ref{pict:tr-2q-hex:metr}.
\begin{figure}[!ht]
\begin{center}
\includegraphics[width=0.45\textwidth]{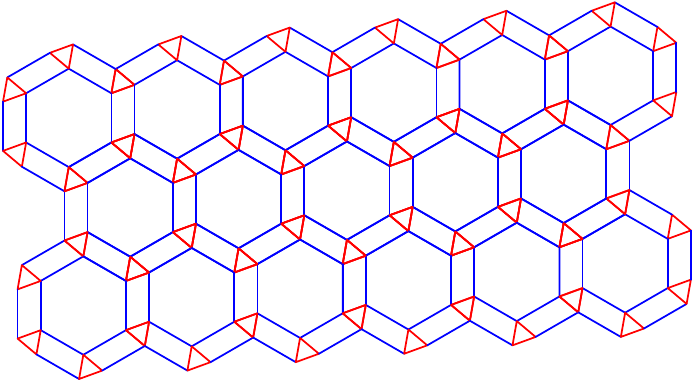}
\hfill
\includegraphics[width=0.45\textwidth]{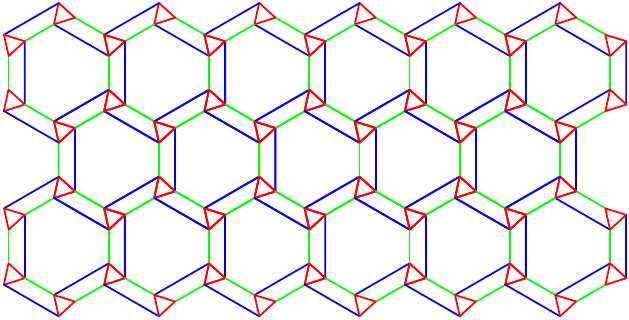}
\caption{Families of monocoronal tilings with one triangle, two quadrilaterals, and one hexagon.}\label{pict:tr-2q-hex:final}
\end{center}
\end{figure}

Proceeding in this manner for all possible cases yields the list of 
tilings in Appendix \ref{class-facetoface}.

\subsection{Non face-to-face tilings}

In the case of non-face-to-face tilings we can prove a lemma similar to Lemma 
\ref{lem:combtype-nonftf}. We assume that each vertex $x$ is contained in the 
relative interior of the edge of some polygon $P_x$. It is clear that there is 
exactly one such polygon.

\begin{lem} \label{lem:combtype-nonftf}
If there is a monocoronal tiling $\mathcal{T}$ such that every vertex is a vertex of polygons with number of edges $\{a_1,\ldots, a_n\}$ and lies in a side of one more polygon then \[\sum \frac{1}{a_i}=\frac{n-1}{2}.\]
\end{lem}

\begin{proof}
Let $k_i$ denote the number of $i$-gons (not including $P_x$) in the corona of the vertex $x$ in a monocoronal tiling $\mathcal{T}$.
Consider the average sum of angles at every vertex, excluding the contribution of  $P_x$. On one hand, it is equal to $\pi$ since $\mathcal{T}$ is a tiling and $P_x$ contributes an angle equal to $\pi$. On the other hand, every $i$-gon adds an angle $\frac{i-2}{i}\pi$ in average. After summing up contributions of all polygons and dividing by $\pi$ we 
obtain 
\[  \sum\limits_{i\geq 3} \frac{k_i(i-2)}{i}=1.\]
Reformulating this with respect to $a_i$ yields the claim.
\end{proof}

One can easily check that the equation in Lemma \ref{lem:combtype-nonftf} is 
possible only in the following cases (we will list only non-zero $k_i$'s):

\[ n=3: \; k_3=3, \quad \quad n=2: \; k_3=1, k_6=1, \quad \mbox{or} \quad n-2: \; k_4=2. \]

Now we can use the same technique as in \ref{subsec:ftf} for obtaining the full list of all possible monocoronal non face-to-face tilings, see Appendix \ref{class-nonfacetoface}.

\subsection{Proofs of the main theorems}

The full list of tilings in Appendices \ref{class-facetoface} and
\ref{class-nonfacetoface} allows us to 
complete the proofs of Theorems \ref{thm:e2-dir} and \ref{thm:e2-cong}.

\begin{proof}[Proof (of Theorem \ref{thm:e2-dir})]
The first assertion can be checked by going through the list in Appendix 
\ref{class-facetoface}, omitting those tilings that require reflected coronae
(i.e., consider only those figures with (D) in the caption).
We cannot find each group directly depicted in the figures. But
many figures cover an entire class of special cases. For instance, Figure
\ref{pict:6tr-5seg} --- showing a triangle tiling with five different 
edge lengths --- covers also the special cases of four, three, two and
one different edge lengths. In particular, the tiling in Figure
\ref{pict:6tr-5seg} with all edge lengths equal shows the regular 
triangle tiling by equilateral triangles with symmetry group $\ast632$.
In this flavor, the following table shows the number of a figure in which
the corresponding symmetry group occurs, together with some additional 
constraint if necessary.
\begin{table}[!ht]
\begin{tabular}{|l|l|}
\hline
$\ast$632 & \ref{pict:6tr-5seg}, all edge lengths equal\\
$\ast$442 & \ref{pict:4quad-rectangles}, all edge lengths equal\\
$\ast$333 & \ref{pict:2tr-2hex}, all interior angles of the hexagons equal\\
$\ast$2222 & \ref{pict:4quad-rectangles}, all rectangles congruent\\
632 & \ref{pict:4tr-hex}\\
442 & \ref{pict:3tr-2q:non-cons1}\\
333 & \ref{pict:6tr:3seg-3pairs}\\
2222 & \ref{pict:6tr-5seg}\\
4 $\ast$ 2 & \ref{pict:3tr-2q:non-cons2}, with squares and half-squares\\
3 $\ast$ 3 & \ref{pict:6tr:3seg-3pairs}, all triangles isosceles (red=blue)\\
2 $\ast$ 22 & \ref{pict:3tr-2q-per}, with squares and equilateral triangles\\
22 $\ast$ & \ref{pict:6tr-5seg}, with red=violet, green=yellow.\\
\hline
\end{tabular}
\caption{Symmetry groups with the number of their corresponding figures and
further specifications.}
\end{table}
To see that this table is complete, one has to verify that the groups
$22\times$, $\ast \ast$, $\ast \times$, $\times \times$ and $\circ$
are not symmetry groups of the tilings in the list (again only considering
figures with (D) in the caption). The simplest way is to check
that all these tilings do have a centre of 2-fold or 3-fold rotation.
This rules out $\ast \ast$, $\ast \times$, $\times \times$ and $\circ$.
Then one may check which figures show glide reflections, but no reflections:
this is true for none of the figures considered here. 
This rules out $22\times$ and shows all assertions except the last claim. 

It remains to check the figures (again only the ones with (D) in the 
caption) for vertex transitivity.
This is done easily since all tilings considered are crystallographic: 
only the vertices in a fundamental  domain of the translation group have 
to be checked for equivalence with respect to the symmetry group. In most cases an
appropriate rotation will suffice. To give an example, consider
Figure \ref{pict:6tr:3seg-3pairs}: The translations in the symmetry group 
of the depicted tiling act transitively on the small triangles. A
3-fold rotation about the centre of a small triangle permutes the vertices
of this small triangle cyclically. Any vertex in the tiling is the
vertex of some small triangle. Thus a composition of some translation and 
some rotation --- which is again a rotation --- maps any given vertex to 
any different vertex.
\end{proof}

\begin{proof}[Proof (of Theorem \ref{thm:e2-cong})]
The assertion can be checked by going through the full list in Appendix 
\ref{class-facetoface}. In the figures one finds 13 crystallographic groups 
as symmetry groups of the depicted tilings (the 12 groups from the last 
proof together with $22 \times$).  
The following short list gives for each group $G$ occurring
the number of one or two figures in the appendix such that $G$ 
is the symmetry group of the corresponding tiling.
\begin{table}[!ht]
\begin{tabular}{|ccccccccccccc|}
\hline
$\ast$632 & $\ast$442 & $\ast$333 & $\ast$2222 & 632 & 442 & 333 & 
2222 & 4 $\ast$ 2 & 3 $\ast$ 3 & 2 $\ast$ 22 & 22 $\ast$ & 22$\times$\\
\hline
\ref{pict:quad-hex-tw} & 
\ref{pict:quad-2oct-alldiff2} & 
\ref{pict:3hex-alldiffalt} & 
\ref{pict:4quad-rectangles} & 
\ref{pict:4tr-hex}, \ref{pict:tr-2q-hex:equal} & 
\ref{pict:3tr-2q:non-cons1}, \ref{pict:4quad:2sq-2par} & 
\ref{pict:6tr:3seg-3pairs}, \ref{pict:2tr-2hex} & 
\ref{pict:6tr-5seg}, \ref{pict:6tr:4seg-2pairs1} & 
\ref{pict:4quad:sq-rect-2trap}, \ref{pict:quad-2oct-egnonalt} & 
\ref{pict:tr-2q-hex:notequal}, \ref{pict:tr-2tw-nonalt} &
\ref{pict:4quad:2rect-2trap}, \ref{pict:non-f2f-4q-two-rect} & 
\ref{pict:4quad:2trap}, \ref{pict:3hex-alldiffother} & 
\ref{pict:6tr:4seg-2pairs1}, \ref{pict:4quad-singletile reflected}\\
\hline
\end{tabular} 
\caption{Symmetry groups with the number of their corresponding figures}
\end{table}
Note that all symmetry groups can be realized with face-to-face tilings.

The remaining three crystallographic groups $\ast \ast$, $\times \times$ 
and $\circ$ are realised by families of tilings that allow also 
non-crystallographic tilings. This is true for the tilings
depicted in Figures \ref{pict:6tr:4seg-2pairs2}, \ref{pict:3tr-2q:non-per1},
\ref{pict:3tr-2q:non-per2}, \ref{pict:4quad-nonperiodic}, 
\ref{pict:non-f2f-4q-one-rect}, and \ref{pict:non-f2f-3tr-isosceles}.

In the four cases where the tilings are face-to-face, non-crystallographic 
symmetry groups can occur since
the tilings consist of alternating layers such that each second layer
is mirror symmetric. The symmetric layers consist either of isosceles 
triangles ($I$) or of rectangles ($R$), as depicted in Figure \ref{pict:nc-layers}.
\begin{figure}[!ht]
\includegraphics[width=0.7\textwidth]{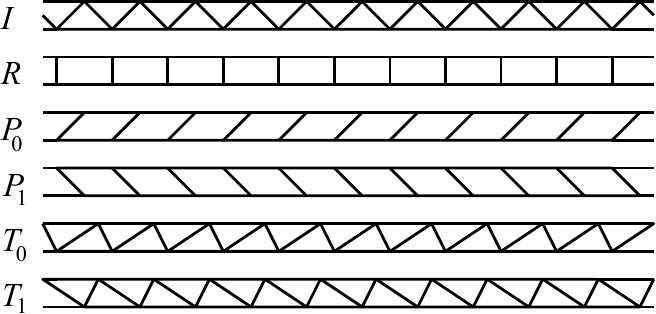}
\caption{Layers of isosceles triangles ($I$), layers of rectangles ($R$),
layers of parallelograms ($P_i$) and layers of non-isosceles 
triangles ($T_i$). \label{pict:nc-layers}}
\end{figure}
The other layers come in two chiral versions and consist either of 
parallelograms ($P_0$ and $P_1$) or of triangles ($T_0$ and $T_1$), see
Figure \ref{pict:nc-layers}. The tilings corresponding to Figure  
\ref{pict:6tr:4seg-2pairs2} consist of alternating layers of the form 
\[\ldots I,T_{i_{-1}},I,T_{i_0}, I,T_{i_1},I,T_{i_2},\ldots, \quad (i_k \in \{0,1\}), \]
the tilings corresponding to Figure  \ref{pict:3tr-2q:non-per1} consist of 
alternating layers of the form  
\[\ldots I,P_{i_{-1}},I,P_{i_0}, I,P_{i_1},I,P_{i_2},\ldots, \quad (i_k \in \{0,1\}), \]
the tilings corresponding to Figure  \ref{pict:3tr-2q:non-per2} consist of 
alternating layers of the form 
\[\ldots R,T_{i_{-1}},R,T_{i_0},I,P_{i_1},I,P_{i_2},\ldots, \quad (i_k \in \{0,1\}), \]
and the tilings corresponding to Figure \ref{pict:4quad-nonperiodic} consist
of alternating layers of the form 
\[\ldots R,P_{i_{-1}},R,P_{i_0},R,P_{i_1},R,P_{i_2},\ldots,  \quad (i_k \in \{0,1\}).\] 
In all cases the sequence $i:=(i_k)_{k \in \Z}$
can be chosen arbitrarily in $\{0,1\}^{\Z}$. The symmetry group
of the corresponding tilings depend on the sequence $i$. In particular,
if $i$ is a non-periodic symbolic sequence, the symmetry group of the
corresponding tiling has translations only parallel to the layers.
(The symbolic sequence $i=(i_k)_{k \in \Z}$ is called {\em non-periodic}, 
if $i_k=i_{k+m}$ for all $k \in \Z$ implies $m=0$.)

By choosing a periodic sequence $i$ appropriately tilings with symmetry 
groups $\ast \ast$, $\times \times$ and $\circ$ can be realised. 
To construct concrete examples we use tilings by parallelograms
and rectangles.
These tilings are depicted in Appendix \ref{no-rotations}. For instance,
$i= \ldots 101100 \; 101100 \; 101100 \ldots$ yields tilings with symmetry
group $\ast \ast$ (no rotations, two reflections in mirror axes parallel
to the layers, see Figure \ref{pict:star-star}). Choosing  
$i= \ldots 111011000100 \;  111011000100 \;  111011000100 \ldots$ 
yields tilings with symmetry group $\times \times$ (no rotations, 
glide reflections with mirror axes orthogonal
to the layers, no reflections, see Figure \ref{pict:norefl-norot}).
With $i= \ldots 0001011 \;  0001011 \;  0001011 \ldots$ 
we obtain tilings with symmetry group $\circ$
(translations only, see Figure \ref{pict:only-transl}).

In order to see that these examples do actually have the claimed
symmetry groups, it is instructive to consider how certain isometries
act on the symbolic sequence $i$. Here $i$ denotes the sequence of the 
original tiling, then $i'$ is the sequence of the tiling after applying
the isometry.
\begin{itemize}
\item Reflection in a line orthogonal to the layers switches the symbols 
1 and 2: $i'_{k} \equiv i_{k} +1 \mod 2$.  
\item Reflection in a line parallel to the layers switches the symbols 
1 and 2 and reverses the order: $i'_{k} \equiv i_{-k} +1 \mod 2$.  
\item Rotation by $\pi$ (around the mid-point of a short edge of some tile 
in layer $0$) reverses the order: $i'_k = i_{-k}$. 
\item Translation orthogonal to the layers by $2m$ layers shifts by $m$:
$i'_{k}=i_{k+m}$. 
\end{itemize} 
In particular, if the sequence $i$ is not invariant under any operation
mentioned above then the corresponding tiling is not invariant under
the corresponding isometry. Keeping this in mind it is easy to check that 
the tilings in Figures \ref{pict:star-star}, \ref{pict:norefl-norot} and 
\ref{pict:only-transl} have the claimed symmetry groups.

It remains to rule out tilings with symmetry groups $\ast \times$ in
the cases considered here, i.e., in the six families of tilings that
may have non-crystallographic symmetry groups. Since these tilings
consist of the layers discussed above, any possible reflection mapping
the tiling to itself has to map the layers to themselves. Thus its mirror 
axis is either parallel to the layers, or orthogonal to the layers. A 
reflection orthogonal to the layers maps a layer $P_1$ to a layer $P_2$,
resp. a layer $T_1$ to a layer $T_2$, thus it cannot be a symmetry
of the tiling. Hence any possible reflection must have a mirror axis
parallel to the layers. 

Since the reflection maps entire layers to entire layers, the mirror
axis lies either on the boundary of some layer, or in the central axis
of some layer. In the first case, the reflection would switch a 
symmetric layer ($I$ or $R$) with a non-symmetric layer ($P_i$ or $T_i$),
hence it is not a symmetry of the tiling. In the second case, it cannot
be the central axis of a layer $P_i$, or $T_i$, or $I$, since these are not
mirror symmetric with respect to reflection about their central axis. Thus the
mirror axis of a possible reflection is the central axis of a layer
$R$.

From the list of the 17 wallpaper groups we
obtain that the mirror axis of a possible glide reflection is
parallel to the axis of the mirror reflection in the symmetry group
$\times \ast$. Hence the axis of any glide reflection is parallel to the layers.
By the same reasoning as above, it must be the central axis of
a layer $R$ consisting of rectangles. It remains to show that there 
is no ``original'' such glide reflection, in terms of group generators. 
I.e., we have to show that any such glide reflection is a combination
of a reflection $r$ and a translation $t$ that are already symmetries
of the tiling. 

The translational part of any glide reflection as above maps rectangles
to rectangles. Thus its length is an integer multiple of the length of one edge
of the rectangles in layer $R$. But a translation by just one edge length
of the rectangles already maps the tiling to itself. Thus for any possible
glide reflection $g$ which translational part shifts by $k$ rectangles,
there is a translation $t$ --- shifting by $-k$ rectangles --- such that
$t \circ g$ is a proper reflection that maps the tiling to itself. 

Hence a group $\ast \times$ cannot occur as symmetry group of a 
monocoronal tiling. 

Now we will establish all possible symmetry groups of 
1-periodic monocoronal tilings. First we will give examples as sequences 
of layers of rectangles and parallelograms as before (sequences of 0's 
and 1's). For the group $\infty\infty$ we can use the sequence $\ldots 0001101000\ldots$ 
with only three 1's. For the group $\infty\ast$ we can use the sequence 
$\ldots 00001111 \ldots$ which is infinity with 0's in one direction and 1's in the other.
And for the group $22\infty$ we can use the sequence $\ldots000010000\ldots$ with only one 1.

The group $\infty\times$ can occur as a symmetry group 
of tiling with layers of isosceles triangles and parallelograms with 
sequence $\ldots 00001111 \ldots$ which is infinity with 0's in one 
direction and 1's in the other. Here the axis of the glide reflection 
is in the middle line of layers of triangles between 0 and 1.

All other three frieze groups $\ast\infty\infty$, 
$2\ast\infty$, and $\ast 22\infty$ can not be realized as symmetry group 
of $1$-periodic monocoronal tiling because they contain a mirror symmetry 
in the direction orthogonal to layers which is impossible for both types 
of non-symmetric layers.
\end{proof}

\section{Euclidean space of dimension 3 and higher} \label{sec:highdim}

In the last section we have seen that monocoronal tilings up
to rigid motions are always 2-periodic, whereas monocoronal tilings 
in general are either 1-periodic or 2-periodic.
Thus one may ask for the possible dimensions of the translation group 
of a monocoronal tiling in higher dimension $d \ge 3$. 
Trivially, the maximal dimension is always $d$, realised for instance
by the canonical face-to-face tiling of $\R^d$ by unit hypercubes.
Thus this section gives upper bounds for the minimal possible dimension 
of translation groups of monocoronal tilings in $\R^d$ for $d \ge 3$,
distinguishing the cases of face-to-face (Theorem \ref{thm:f2f-minper}) vs 
non face-to-face (Theorem \ref{thm:nonf2f-minper}). Both case split
further with respect to direct congruence vs congruence. In analogy to the plane 
case, two sets $A,B \subset \R^d$ are called {\em congruent}, 
if there is $t \in \R^d$ such that $A = t+RB$ for some $R \in O(d)$.
$A$ and $B$ are called {\em directly congruent}, 
if there is $t \in \R^d$ such that $A = t+RB$ for some $R \in SO(d)$.
The latter corresponds to the existence of a rigid motion of $\R^d$ moving $A$ to
$B$.

\begin{thm}\label{thm:f2f-minper}
There are face-to-face tilings of $\R^d$ that are $\lceil \frac{d}{2}\rceil$-periodic
and monocoronal.\\
There are face-to-face tilings of $\R^d$ that are $\lceil \frac{d+1}{2}\rceil$-periodic 
and monocoronal up to rigid motions. 
\end{thm}

\begin{proof}
For the first claim, 
consider a 1-periodic monocoronal tiling $\mathcal{T}$ of the Euclidean plane  
(e.g. Figure \ref{pict:3dim-block}. Consider the direct 
product $\mathcal{T}\times\ldots
\times \mathcal{T}$ of $\lfloor \frac{d}{2}\rfloor$ copies of $\mathcal{T}$. 
This yields a $\lfloor \frac{d}{2}\rfloor$-periodic monocoronal tiling. 
If $d$ is odd then we need to take 
one additional direct product with some 1-periodic tiling of $\R^1$ (for example 
a tiling of the line by unit intervals).

For the second claim, note that the Cartesian product $A \times B$ is directly
congruent if either $A$ or $B$ is directly congruent. 
Let $\T$ be the direct product of a 1-periodic tiling 
from Figure \ref{pict:4quad-nonperiodic} with a tiling of $\R^1$ by unit 
intervals. Then $\T$ is a monocoronal tiling up to rigid motions 
of $\R^3$. The product of $\T$ with $k$ copies of any 1-periodic monocoronal
tiling of $\R^2$ is a monocoronal tiling up to 
rigid motions of $\R^{3+2k}$ that is $2+k$-periodic.
If $d$ is odd, then $d=3+2k$ for some $k$, and $\T$ is $\lceil \frac{d+1}{2} \rceil$-periodic.
If $d$ is even then consider the additional direct product
with a further tiling of $\R^1$ by unit intervals.
This yields a $\lceil \frac{d+1}{2}\rceil$-periodic monocoronal
tiling up to rigid motions of $\R^d$ for all $d \ge 3$.
\end{proof}

For the non face-to-face case one can even push it further: one can construct 
non-periodic tilings in $\mathbb{R}^d$ for $d \ge 3$.

\begin{thm}\label{thm:nonf2f-minper}
For any $d\geq 3$ there are non-periodic non face-to-face tilings of $\R^d$ 
that are monocoronal.\\
For any $d\geq 4$ there are non-periodic non face-to-face tilings of $\R^d$
that are monocoronal up to rigid motions .\\
\end{thm}

\begin{proof}
We will start with the first claim. We will show the construction of such an 
example in $\R^3$. It can easily be generalised to higher dimensions.

We start with a 1-periodic tiling from Figure \ref{pict:3tr-2q:non-per2} 
where all rectangles are unit squares and all triangles are isosceles 
right triangles with edges $1,1,\sqrt2$ (see Figure \ref{pict:3dim-block}).
\begin{figure}[!ht]
\begin{center}
\includegraphics[width=0.7\textwidth]{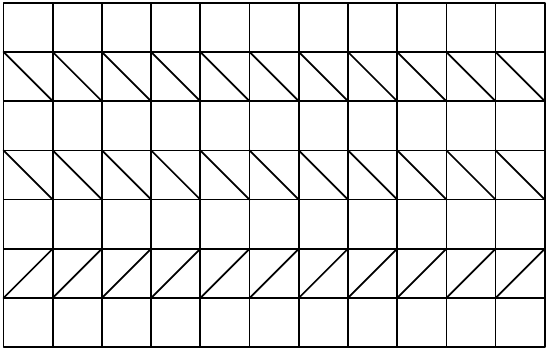}
\caption{Building block for an non-periodic tiling.}
\label{pict:3dim-block}
\end{center}
\end{figure}
Using this tiling we can create a {\it layer} $L$ by taking its direct 
product with an orthogonal unit interval (see Figure \ref{pict:3dim-1layer}).
This layer is 1-periodic.
\begin{figure}[!ht]
\begin{center}
\includegraphics[width=0.7\textwidth]{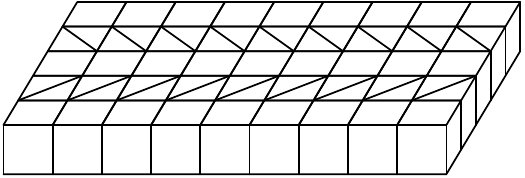}
\caption{Layer of unit cubes and triangular prisms.}
\label{pict:3dim-1layer}
\end{center}
\end{figure}
In the next step we add a layer of unit cubes below the layer $L$ (see Figure 
\ref{pict:3dim-2layers}). This is the point where our tiling starts to be 
non face-to-face.
\begin{figure}[!ht]
\begin{center}
\includegraphics[width=0.7\textwidth]{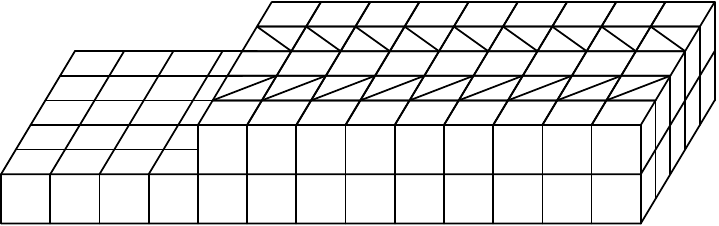}
\caption{Layer $L$ and layer of unit cubes.}
\label{pict:3dim-2layers}
\end{center}
\end{figure}
In the next step (see Figure \ref{pict:3dim-3layers}) we add one more layer 
congruent to $L$ below that is rotated by $\frac{\pi}{2}$ with respect to $L$. This latter
layer is 1-periodic, but in a direction orthogonal to the period of
the first layer. Hence the union of all three layers is non-periodic.
\begin{figure}[!ht]
\begin{center}
\includegraphics[width=0.8\textwidth]{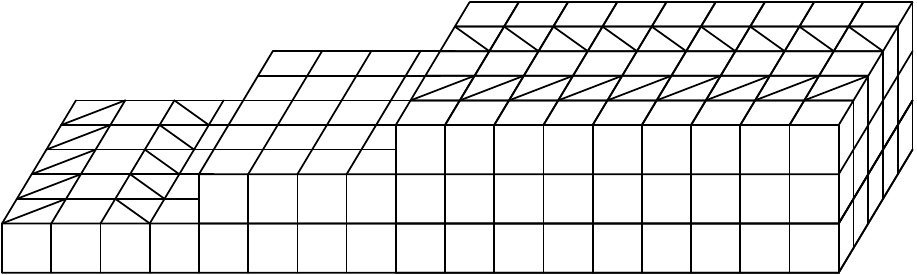}
\caption{Three layers of unit cubes, two of them contains cubes cut into prisms.}
\label{pict:3dim-3layers}
\end{center}
\end{figure}
Proceeding in this manner yields a tiling of $\R^3$ such that even layers 
consist of unit cubes, odd layers consist of unit cubes and triangular 
prisms. It is clear that any vertex corona consists of four unit cubes 
from some even layer, two unit cubes from an adjacent odd layer, and 
three prisms. Hence the tiling is a monocoronal tiling.

In each step of the construction one may choose a layer that is a 
translate of $L$, or a copy of $L$ rotated by $\frac{\pi}{2}$. 
If the sequence of choices is non-periodic then the tiling has no
period in the vertical direction. Hence we obtain a non-periodic tiling.

The constructed tiling is not a monocoronal tiling up to 
rigid motions, so for the second claim of the theorem we need to modify the 
construction a little bit. As for the first claim, we will present a construction 
for $\R^4$ that can be easily generalised to higher dimensions.

Again, we start from the tiling by unit squares and right isosceles triangles 
in Figure \ref{pict:3dim-block}. In a preliminary step we take the 
direct product of this tiling with a tiling of the orthogonal line by unit intervals.
This yields a $2$-periodic tiling of $\R^3$ where each vertex corona consists
of unit cubes and triangular prisms, and each vertex corona is mirror-symmetric.

Taking the direct product of this tiling with an orthogonal unit interval yields  
a layer $L'$ in $\R^4$ consisting of unit $4$-cubes and prisms over triangular 
prisms. Now we can repeat the steps from the proof of the first claim of Theorem
\ref{thm:f2f-minper}. 
Every even layer of the tiling consists of unit $4$-cubes, and every odd layer is
a directly congruent copy of the layer $L'$, preserving the ``almost'' cubical 
structure on its boundary.

The initial 3-dimensional tiling was monocoronal, with mirror-symmetric corona.
Thus the resulting tiling is also monocoronal with mirror-symmetric corona, 
hence it is a monocoronal tiling up to rigid motion. We can force this 
tiling to be non-periodic by taking some non-periodic sequence of 
rotations of $L'$ for odd layers. 
$L'$ is $2$-periodic so in every step we can choose one direction in which 
the copy of $L'$ is non-periodic. It suffices to choose each of three 
directions at least once to destroy any period parallel to the layers.
\end{proof}

\section{Conclusion} \label{sec:concl}

This short conclusion mentions some open problems in the context of this  
paper that are still open and may suggest further work. 

The smallest possible 
dimension of the translation group of monocoronal face-to-face tilings of 
Euclidean spaces of dimension at least $3$ is still unknown: it is 
0, 1 or 2. In particular, the question ``Does there exist a non-periodic 
monocoronal face-to-face tiling in $\R^3$'' is still open.

Throughout this paper we restrict our study only to convex tiles, but the 
same question could be asked for tilings with non-convex polygons 
(resp.\ polytopes) as well. Allowing non-convex polygons does not change the 
classification of two-dimensional face-to-face monocoronal tiling in 
Appendix~\ref{class-facetoface}: every possible monocoronal tiling
using non-convex polygons is already covered by the classification. 
For non face-to-face tiling it is even easier, because in this case 
no angle of any polygon can be greater than $\pi$, hence no monocoronal
non-face-to-face tiling with non-convex polygons is possible at all.

One may also ask for classifications of monocoronal tilings in other
spaces of constant curvature, in particular in hyperbolic space.
For instance we might ask:
Is it true that every monocoronal tiling of $\mathbb{H}^d$ is crystallographic?
It is easier to answer this question for hyperbolic spaces than for Euclidean 
spaces since there is a family of non-crystallographic tilings of $\mathbb{H}^d$ 
with unique tile corona. This tiling can be used to construct tilings with 
unique vertex corona, that is, monocoronal tilings.

\begin{thm} \label{thm:hyp}
There is a non-crystallographic face-to-face tiling of $\mathbb{H}^d$ that
is monocoronal up to congruence.
\end{thm}

\begin{proof}
Here we show the construction for $\mathbb{H}^2$. An analogous construction 
works for arbitrary dimension.

We start from (one of) B\"or\"oczky tiling $\mathcal{B}$ \cite{Bor}. It is a 
non-crystallographic tiling of hyperbolic plane \cite[Theorem. 4.4]{DF} by equal 
pentagons. Figure \ref{pict:hyper-1} shows a schematic view of this 
tiling as a tiling of the representation of $\mathbb{H}^2$ as lower half plane.
\begin{figure}[!ht]
\begin{center}
\includegraphics[width=0.7\textwidth]{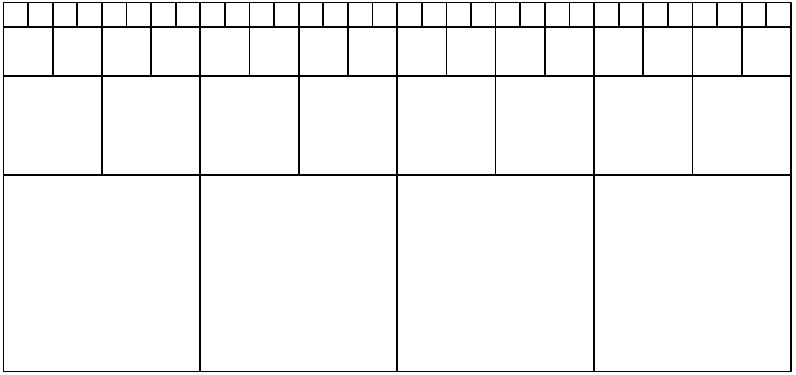}
\caption{Example of B\"or\"oczky tiling.}
\label{pict:hyper-1}
\end{center}
\end{figure}

It is easy to see that this tiling $\mathcal{B}$ is not monocoronal. But 
every tile is surrounded by other tiles ``in the same way'', 
so we can use it to construct a monocoronal tiling. We construct the 
{\it dual tiling} by taking barycenters of the initial tiles as vertices 
of a new tiling and new tiles are convex hulls of vertices corresponding 
to ``old'' tiles incident to one ``old'' vertex.
\begin{figure}[!ht]
\begin{center}
\includegraphics[width=0.7\textwidth]{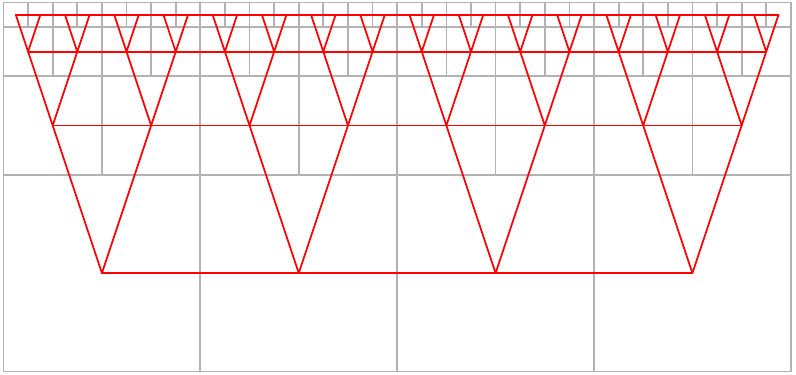}
\caption{Dual to B\"or\"oczky tiling.}
\label{pict:hyper-2}
\end{center}
\end{figure}

For arbitrary tilings this construction does not necessary yield a (face-to-face) 
tiling, but in the case of B\"or\"oczky tiling it works. Barycenters 
of tiles of one ``horizontal layer'' lie on one horocycle (horosphere 
in $\mathbb{H}^d$). Thus tiles of the dual tiling form a layer 
structure between neighboring horocycles. Moreover, this tiling 
is a monocoronal tiling, since in the original tiling all first tile-coronae
are congruent. The dual tiling is non-crystallographic, since the initial 
B\"or\"oczky tiling was non-crystallographic.
\end{proof}

Theorem \ref{thm:hyp} shows that face-to-face tilings in $\HH^d$ that are 
monocoronal up to congruence can be non-crystallographic (in a pretty 
strict sense: their symmetry group being finite, see \cite{DF}). They can 
also be crystallographic for small $d$, any regular tiling of $\HH^d$
yields an example.

The same question with respect to monocoronal tilings up to rigid motions 
is still open: The vertex coronae of the tilings in Figure \ref{pict:hyper-2}
are congruent, but not directly congruent. It might also be interesting to
study the situation in the higher dimensional analogues of the B\"or\"oczky
tilings. 

\section*{Acknowledgements}
The research of the second author is partially supported by a grant of the 
Dynasty Foundation and by the grant ``Leading Scientific Schools'' NSh-4833.2014.1.

\newpage

\appendix

\section{Two-dimensional face-to-face monocoronal tiling}\label{class-facetoface}

In this section we will list all possible monocoronal tilings of the 
Euclidean plane. The tilings are grouped with respect to
(a) the number of polygons incident to the central vertex and (b) the type 
of this polygons. The caption of any figure will contain (NC) if
the depicted tilings may have a non-crystallographic symmetry group, i.e.,
whether there is a 1-periodic tiling in the class of depicted tilings.
The caption will contain (D) if all vertex coronae are directly congruent.
Note, that if the vertex corona is mirror symmetric, then all vertex corona
are directly congruent.

\subsection{Six triangles}

We will divide this case further into subcases with respect to the number 
of different edge lengths occurring. It is easy to see that the 
case with six different edge lengths incident to each vertex is impossible.

\subsubsection{Five edge lengths}

There is only one possible case of that type.
\begin{center}
\includegraphics[width=0.6\textwidth]{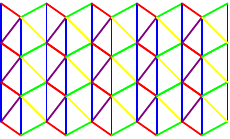}
\captionof{figure}{Tiling with two different non-isosceles non-reflected 
triangles (D).}\label{pict:6tr-5seg}
\end{center}

\subsubsection{Four edge lengths}

First we will list all possible tilings with two pairs of edges of equal length. 
Here and before we will omit families that can appear as limiting cases of 
the previous ones. For example here we will not list the case where 
the yellow edge and the red edge from Figure \ref{pict:6tr-5seg} have
the same length.

\begin{minipage}{0.45\textwidth}
\begin{center}
\includegraphics[width=\textwidth]{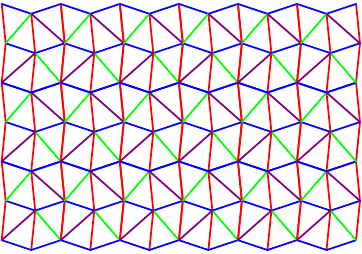}
\captionof{figure}{Tiling with two different types of non-isosceles 
triangles.}\label{pict:6tr:4seg-2pairs1}
\end{center}
\end{minipage}
\hfill
\begin{minipage}{0.45\textwidth}
\begin{center}
\includegraphics[width=\textwidth]{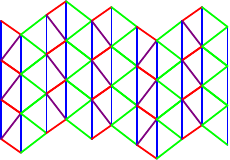}
\captionof{figure}{Tiling with two types of triangles: one is non-isosceles 
and one is isosceles (NC).}\label{pict:6tr:4seg-2pairs2}
\end{center}
\end{minipage}

The only family of this type with three edges of equal length incident to each 
vertex can be obtained from Figure \ref{pict:6tr-5seg} by forcing the 
yellow edge and the blue edge to have the same length.

\newpage

\subsubsection{Three edge lengths}

This part we start from tilings where each vertex is incident to three pairs 
of equal edges.

\begin{center}
\includegraphics[width=0.7\textwidth]{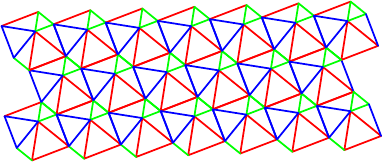}
\captionof{figure}{Tiling with three different regular triangles and one 
type of non-isosceles triangle (D).}\label{pict:6tr:3seg-3pairs}
\end{center}

The second option is the following: each vertex is incident to three edges 
of the first type, two edges of the second type and one edge of the third 
type. But this case is covered by families that we listed before.

The third case where each vertex is incident to four equal edges is also 
covered by previous cases.

\subsubsection{Two edge lengths}

This is also covered by previous families.

\subsubsection{One edge length}

This case is trivial since there is only such tiling, namely the 
canonical tiling by regular triangles.

\subsection{Five polygons: three triangles and two quadrilaterals}

First we will list all possible tilings where quadrilaterals are consecutive 
in every vertex-corona.

\begin{center}
\includegraphics[width=0.7\textwidth]{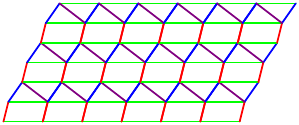}
\captionof{figure}{Tiling with two equal parallelograms and three equal 
triangles (D).}\label{pict:3tr-2q-per}
\end{center}

Two different limiting cases of this corona can generate non-crystallographic 
tilings, see Figure \ref{pict:3dim-block} for one example. 
However, if we combine two limiting cases then we will obtain again a 
crystallographic tiling with rectangles and isosceles triangles.

\newpage

\begin{minipage}{0.45\textwidth}
\begin{center}
\includegraphics[width=\textwidth]{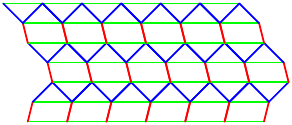}
\captionof{figure}{Tiling with two equal non-rectangular parallelograms 
and with three equal isosceles triangles (NC).}\label{pict:3tr-2q:non-per1}
\end{center}
\end{minipage}
\hfill
\begin{minipage}{0.45\textwidth}
\begin{center}
\includegraphics[width=\textwidth]{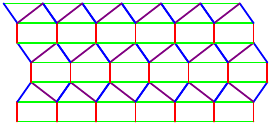}
\captionof{figure}{Tiling with two equal rectangles and three equal 
non-isosceles triangles (NC).}\label{pict:3tr-2q:non-per2}
\end{center}
\end{minipage}

The second family consists of tilings whose coronas contain non-consecutive 
quadrilaterals. There are two different families of such tilings.

\begin{minipage}{0.45\textwidth}
\begin{center}
\includegraphics[width=\textwidth]{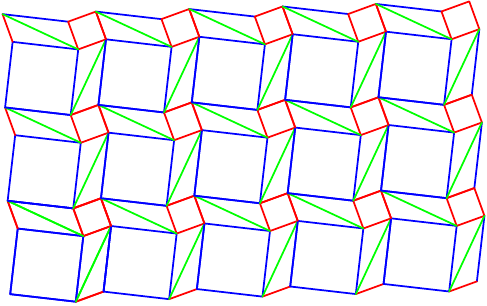}
\captionof{figure}{Tiling with two different squares and with three 
equal triangles (D).}\label{pict:3tr-2q:non-cons1}
\end{center}
\end{minipage}
\hfill
\begin{minipage}{0.45\textwidth}
\begin{center}
\includegraphics[width=\textwidth]{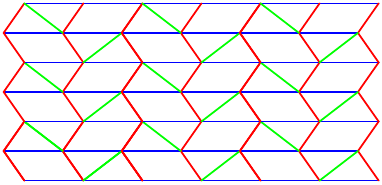}
\captionof{figure}{Tiling with two equal (non-consecutive) parallelograms 
and three equal triangles.}\label{pict:3tr-2q:non-cons2}
\end{center}
\end{minipage}

\subsection{Five polygons: four triangles and one hexagon}

There is only one family of such tilings. The vertex corona consists of
a regular hexagon, one regular triangle and one arbitrary triangle.

\begin{center}
\includegraphics[width=0.8\textwidth]{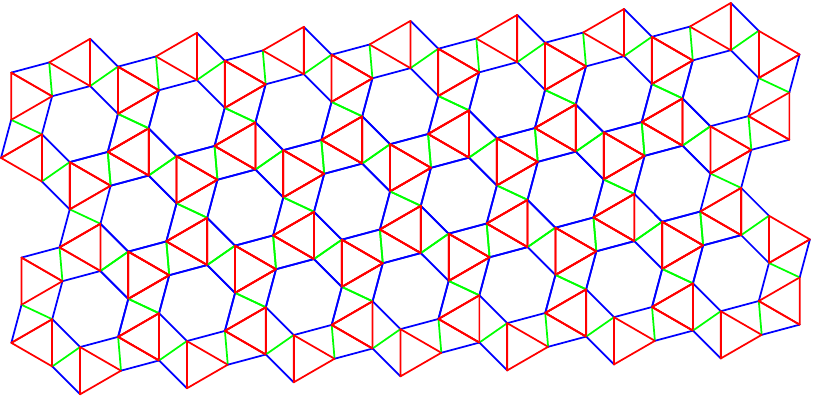}
\captionof{figure}{Tiling with regular hexagon, regular triangle and an 
arbitrary triangle (D).}\label{pict:4tr-hex}
\end{center}

\newpage

\subsection{Four polygons: two triangles and two hexagons}

There is only one family of such tilings. Its vertex corona consists of 
two regular triangles (of different size) and one hexagon with alternating 
edges and angles.

\begin{center}
\includegraphics[width=0.8\textwidth]{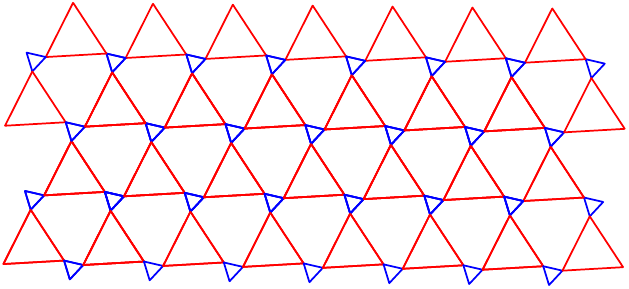}
\captionof{figure}{Tiling with two types of regular triangles and equal 
hexagons (D).}\label{pict:2tr-2hex}
\end{center}

\subsection{Four polygons: triangle, two quadrilaterals and one hexagon}

There are two families of this type. This was shown in Section \ref{sec:plane}. 
The first family has a vertex corona consisting of a hexagon with two different 
edge lengths and the second one consists of hexagons with unique edge length.

\begin{minipage}{0.45\textwidth}
\begin{center}
\includegraphics[width=\textwidth]{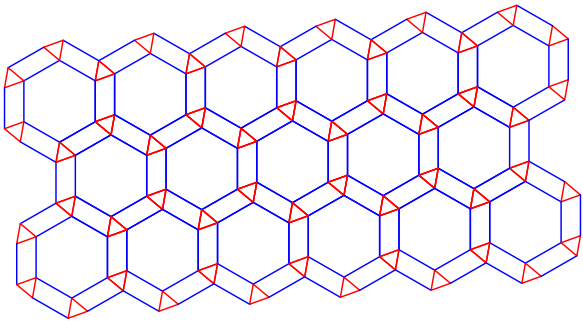}
\captionof{figure}{Tiling with regular triangle, regular hexagon and two 
parallelograms (D).}\label{pict:tr-2q-hex:equal}
\end{center}
\end{minipage}
\hfill
\begin{minipage}{0.45\textwidth}
\begin{center}
\includegraphics[width=\textwidth]{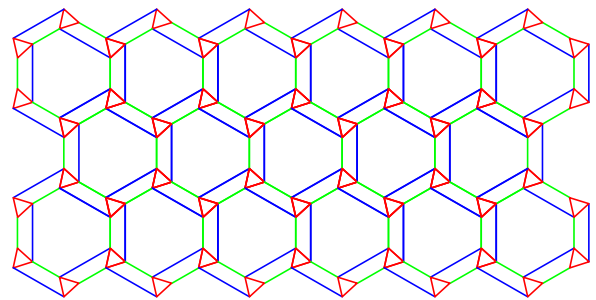}
\captionof{figure}{Tiling with regular triangle, hexagon and two equal 
trapezoids.}\label{pict:tr-2q-hex:notequal}
\end{center}
\end{minipage}

\subsection{Four polygons: four quadrilaterals}

In the first part all edges incident to one vertex will be of pairwise
different length.

\begin{minipage}{0.45\textwidth}
\begin{center}
\includegraphics[width=\textwidth]{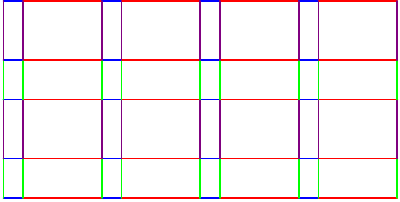}
\captionof{figure}{All quadrilaterals are rectangles.}\label{pict:4quad-rectangles}
\end{center}
\end{minipage}
\hfill
\begin{minipage}{0.45\textwidth}
\begin{center}
\includegraphics[width=\textwidth]{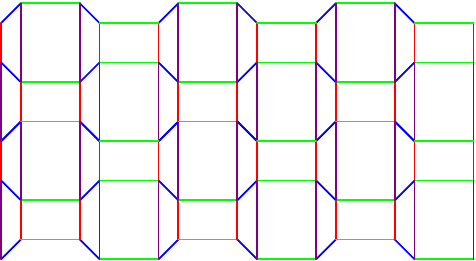}
\captionof{figure}{The vertex corona consists of two rectangles and two 
equal trapezoids.}\label{pict:4quad:2rect-2trap}
\end{center}
\end{minipage}

\begin{minipage}{0.45\textwidth}
\begin{center}
\includegraphics[width=\textwidth]{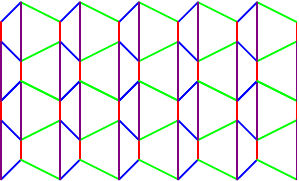}
\captionof{figure}{Tiling with two types of trapezoids.}\label{pict:4quad:2trap}
\end{center}
\end{minipage}
\hfill
\begin{minipage}{0.45\textwidth}
\begin{center}
\includegraphics[width=\textwidth]{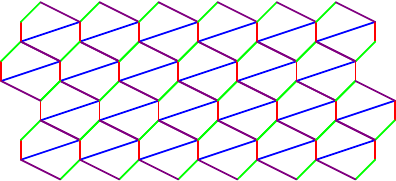}
\captionof{figure}{Tiling with equal quadrilaterals (D).}\label{pict:4quad-singletile}
\end{center}
\end{minipage}

Now we list additional families that appear when we allow two edges incident 
to one vertex to be of equal length (Figures \ref{pict:4quad-singletile 
reflected}, \ref{pict:4quad:sq-rect-2trap}, \ref{pict:4quad-nonperiodic}, 
and \ref{pict:4quad-singletile-other}).

\begin{minipage}{0.45\textwidth}
\begin{center}
\includegraphics[width=\textwidth]{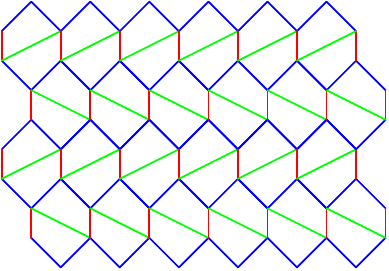}
\captionof{figure}{Tiling with equal quadrilaterals with reflections.}\label{pict:4quad-singletile reflected}
\end{center}
\end{minipage}
\hfill
\begin{minipage}{0.45\textwidth}
\begin{center}
\includegraphics[width=\textwidth]{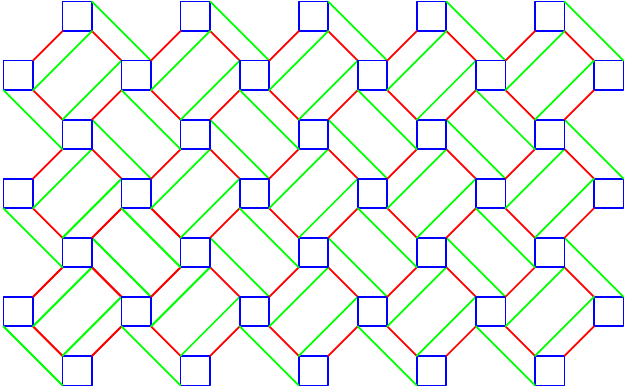}
\captionof{figure}{Tiling with square, rectangle and two trapezoids.}\label{pict:4quad:sq-rect-2trap}
\end{center}
\end{minipage}

\begin{figure}[!ht]
\begin{center}
\includegraphics[width=0.7\textwidth]{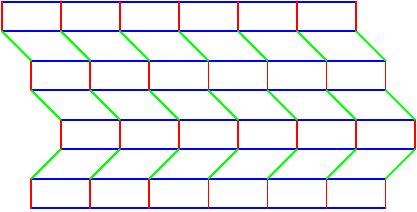}
\caption{A pair of opposite edges in the vertex corona are of the same length
and there are two different tiles. Then the tiles are parallelograms.
(NC) if one of the parallelograms is a rectangle.}\label{pict:4quad-nonperiodic}
\end{center}
\end{figure}

\newpage

\begin{minipage}{0.45\textwidth}
\begin{center}
\includegraphics[width=\textwidth]{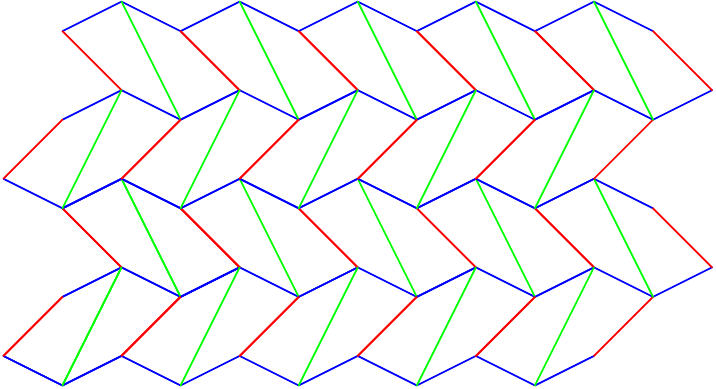}
\captionof{figure}{A pair of opposite edges are equal and all tiles are equal.}\label{pict:4quad-singletile-other}
\end{center}
\end{minipage}
\hfill
\begin{minipage}{0.45\textwidth}
\begin{center}
\includegraphics[width=\textwidth]{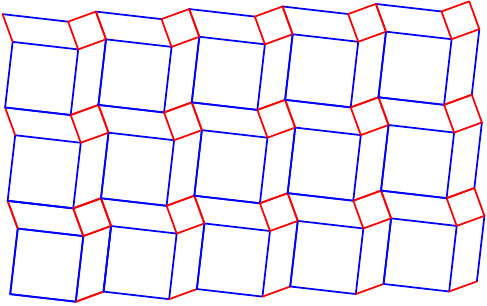}
\captionof{figure}{Tiling with two different squares and equal 
parallelograms (D).}\label{pict:4quad:2sq-2par}
\end{center}
\end{minipage}

The families of tilings with four quadrilaterals at each vertex and with two 
different edge lengths are covered by previous cases, together with the tiling
in Figure \ref{pict:4quad:2sq-2par}. 

All families of tilings with a single edge length are limit cases of the previous
cases.

\subsection{Three polygons: triangle and two $12$-gons}

In this part we will list possible tilings with vertex coronas consisting of 
one triangle and two $12$-gons.

\begin{minipage}{0.45\textwidth}
\begin{center}
\includegraphics[width=\textwidth]{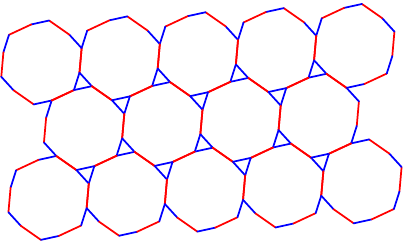}
\captionof{figure}{Angles of both $12$-gons are alternating
(D).}\label{pict:tr-2tw-alt}
\end{center}
\end{minipage}
\hfill
\begin{minipage}{0.45\textwidth}
\begin{center}
\includegraphics[width=\textwidth]{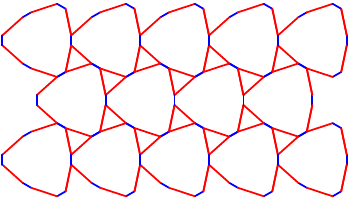}
\captionof{figure}{Some consecutive angles of $12$-gons are equal.}\label{pict:tr-2tw-nonalt}
\end{center}
\end{minipage}

\subsection{Three polygons: quadrilateral, hexagon and $12$-gon}

There is only one family of tilings with vertex corona consisting of 
quadrilateral, hexagon and $12$-gon.

\begin{figure}[!ht]
\begin{center}
\includegraphics[width=0.7\textwidth]{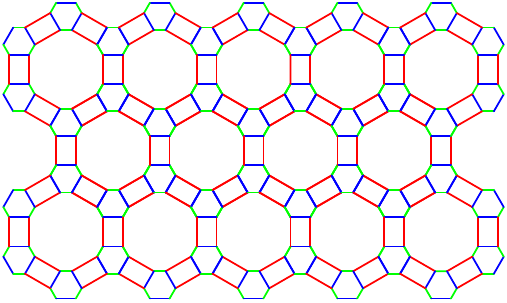}
\caption{Tiling with congruent coronas consisting of quadrilateral, 
hexagon and $12$-gon.}\label{pict:quad-hex-tw}
\end{center}
\end{figure}

\subsection{Three polygons: quadrilateral and two octagons}

Every vertex is incident to three edges, two of them are joint edges 
of a quadrilateral and an octagon. First we list all tilings where these 
two edges are different.

\begin{minipage}{0.45\textwidth}
\begin{center}
\includegraphics[width=\textwidth]{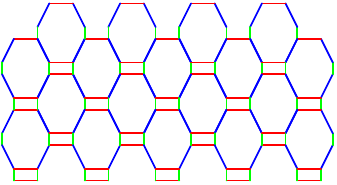}
\captionof{figure}{Octagons has edges from both families green and red.}\label{pict:quad-2oct-alldiff1}
\end{center}
\end{minipage}
\hfill
\begin{minipage}{0.45\textwidth}
\begin{center}
\includegraphics[width=\textwidth]{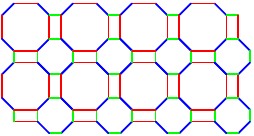}
\captionof{figure}{Each octagon contains edges only from one family either 
red or green.}\label{pict:quad-2oct-alldiff2}
\end{center}
\end{minipage}

The second case contains families with equal edges between quadrilaterals 
and octagons.

\begin{minipage}{0.45\textwidth}
\begin{center}
\includegraphics[width=\textwidth]{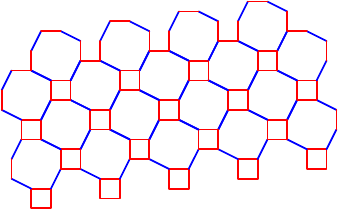}
\captionof{figure}{Angles of both octagons are alternated.}\label{pict:quad-2oct-eqalt}
\end{center}
\end{minipage}
\hfill
\begin{minipage}{0.45\textwidth}
\begin{center}
\includegraphics[width=\textwidth]{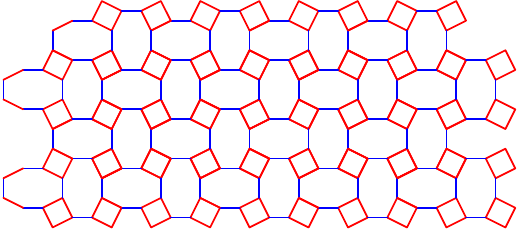}
\captionof{figure}{Some consecutive angles of octagons are equal.}\label{pict:quad-2oct-egnonalt}
\end{center}
\end{minipage}

\subsection{Three polygons: three hexagons} 
In this subsection we present all possible families of tilings with congruent 
vertex coronas consisting of three hexagons. Each family has a short description 
of the most important metrical properties.

The first group contains the tilings where all three edges incident to one vertex 
have different length.

\begin{minipage}{0.45\textwidth}
\begin{center}
\includegraphics[width=\textwidth]{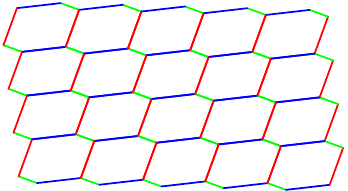}
\captionof{figure}{All hexagons are translations of each other
(D).}\label{pict:3hex-alldifftransl}
\end{center}
\end{minipage}
\hfill
\begin{minipage}{0.45\textwidth}
\begin{center}
\includegraphics[width=\textwidth]{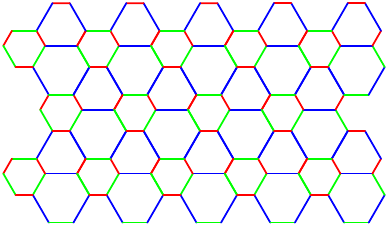}
\captionof{figure}{All hexagons are different and have alternating 
colourings.}\label{pict:3hex-alldiffalt}
\end{center}
\end{minipage}

\begin{figure}[!ht]
\begin{center}
\includegraphics[width=0.5\textwidth]{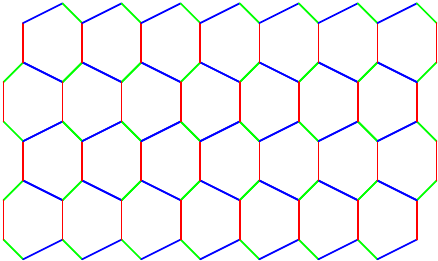}
\caption{Hexagons are the same but odd rows are reflected.}\label{pict:3hex-alldiffother}
\end{center}
\end{figure}

In the sequel we list the possible families with two equal 
edges incident to every vertex. We omit those tilings that can be achieved 
from previously described families (Figures 
\ref{pict:3hex-alldifftransl}--\ref{pict:3hex-alldiffother}) by letting two 
different edges to be of equal lengths without using the additional 
degree of freedom obtained from the possibility of having different angles.

\begin{minipage}{0.45\textwidth}
\begin{center}
\includegraphics[width=\textwidth]{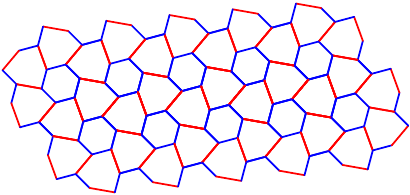}
\captionof{figure}{Two incident edges are equal (D).}\label{pict:3hex-twoequal}
\end{center}
\end{minipage}
\hfill
\begin{minipage}{0.45\textwidth}
\begin{center}
\includegraphics[width=\textwidth]{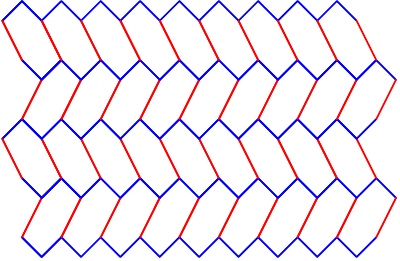}
\captionof{figure}{Two edges are of equal length. This allows using
reflected coronae with respect to the case in Figure 
\ref{pict:3hex-alldifftransl}.}\label{pict:3hex-twoequalnoreg}
\end{center}
\end{minipage}

The cases where all edges are of equal length are obtained from the previous 
cases by letting some different edge lengths to be equal. For example, if we 
force all edges in Figure \ref{pict:3hex-alldiffalt} to be equal then we will 
obtain the tiling by regular hexagons.

\newpage

\section{Non face-to-face monocoronal tiling}\label{class-nonfacetoface}

Here we list all possible families of non face-to-face monohedral tilings.
We will group them with respect to polygons that are incident to each 
vertex (except the one polygon that contributes not a vertex, but an edge).

\subsection{Triangle and hexagon} 
We start from the case where each hexagon is regular.

\begin{minipage}{0.45\textwidth}
\begin{center}
\includegraphics[width=\textwidth]{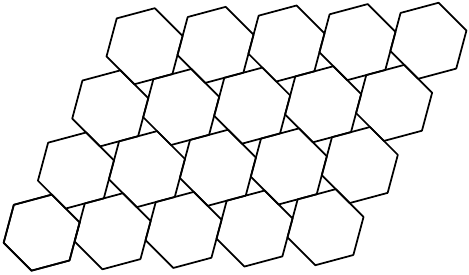}
\captionof{figure}{Tiling with regular hexagon and regular triangle. 
The edge of the hexagon is longer (D).}\label{pict:non-f2f-reg-tr-hex-0}
\end{center}
\end{minipage}
\hfill
\begin{minipage}{0.45\textwidth}
\begin{center}
\includegraphics[width=\textwidth]{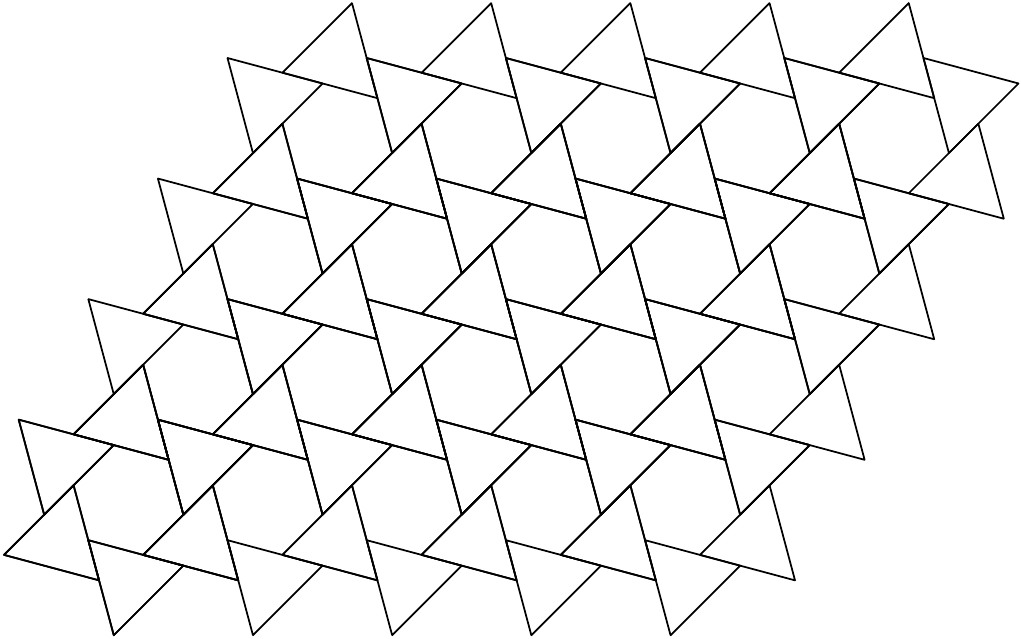}
\captionof{figure}{Tiling with regular hexagon and regular triangle. 
The edge of the hexagon is shorter (D).}\label{pict:non-f2f-reg-tr-hex-1}
\end{center}
\end{minipage}

The second case is when the hexagon is non-regular. In that case any of its 
longer side touches two equal regular triangles.

\begin{figure}[!ht]
\begin{center}
\includegraphics[width=0.6\textwidth]{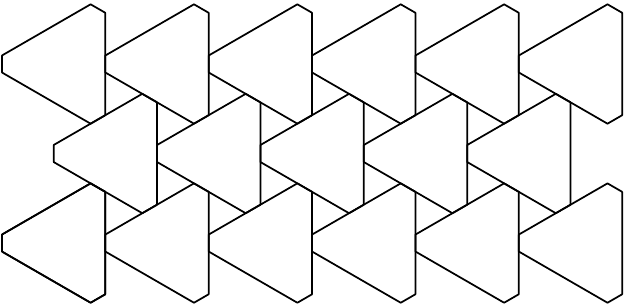}
\caption{Tiling with regular triangle and non-regular hexagon.}\label{pict:non-f2f-tr-nonreghex-1}
\end{center}
\end{figure}

\subsection{Two quadrilaterals} There are two significantly different cases: 
the intersection of two quadrilaterals incident to the vertex is an entire 
edge of both of them, or it is an entire edge of one of them and only a part 
of an edge of the other. We start our list with the first case.

\begin{minipage}{0.45\textwidth}
\begin{center}
\includegraphics[width=\textwidth]{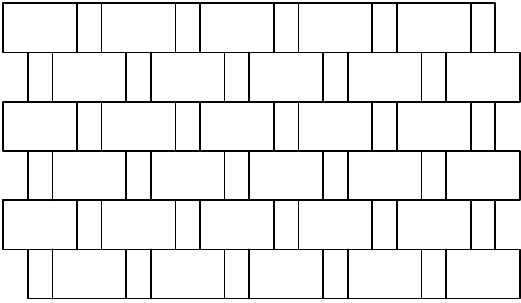}
\captionof{figure}{Tiling with two types of rectangles.}\label{pict:non-f2f-4q-two-rect}
\end{center}
\end{minipage}
\hfill
\begin{minipage}{0.45\textwidth}
\begin{center}
\includegraphics[width=\textwidth]{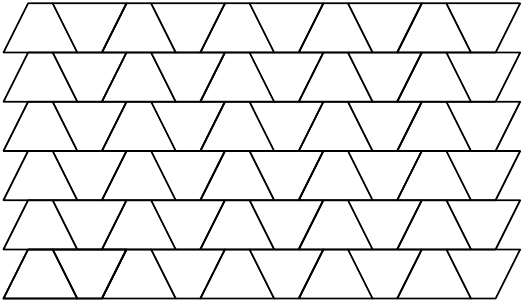}
\captionof{figure}{Tiling with one type of trapezoid.}\label{pict:non-f2f-4q-trap}
\end{center}
\end{minipage}

\begin{minipage}{0.45\textwidth}
\begin{center}
\includegraphics[width=\textwidth]{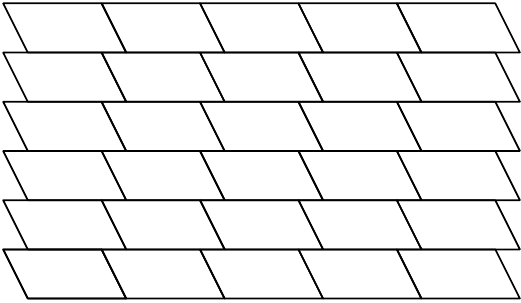}
\captionof{figure}{Tiling with one type of parallelogram (D).}\label{pict:non-f2f-4q-one-par}
\end{center}
\end{minipage}
\hfill
\begin{minipage}{0.45\textwidth}
\begin{center}
\includegraphics[width=\textwidth]{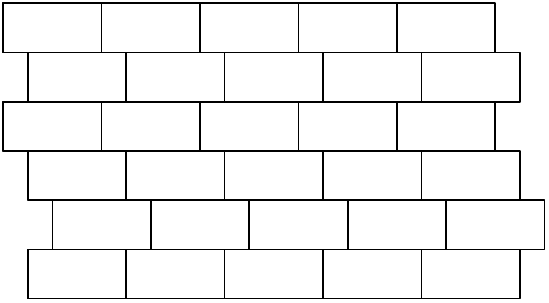}
\captionof{figure}{Tiling with one type of rectangle (NC).}\label{pict:non-f2f-4q-one-rect}
\end{center}
\end{minipage}

And there are two additional tilings for the second case.

\begin{minipage}{0.45\textwidth}
\begin{center}
\includegraphics[width=\textwidth]{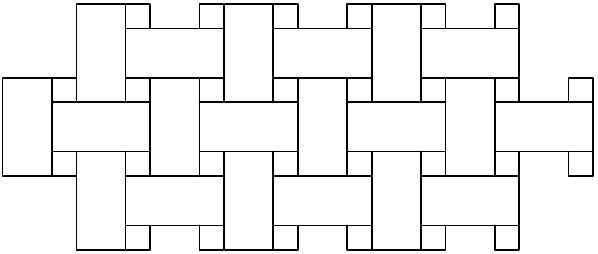}
\captionof{figure}{Tiling with square and rectangle.}\label{pict:non-f2f-4q-sq-rect}
\end{center}
\end{minipage}
\hfill
\begin{minipage}{0.45\textwidth}
\begin{center}
\includegraphics[width=\textwidth]{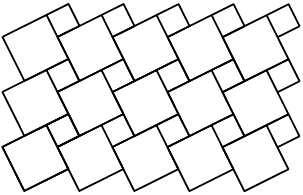}
\captionof{figure}{Tiling with two types of squares (D).}\label{pict:non-f2f-4q-two-sq}
\end{center}
\end{minipage}

\subsection{Three triangles} 
As in the previous case the tilings are combined into families with 
similar incidence structure of the vertex corona.

\begin{minipage}{0.45\textwidth}
\begin{center}
\includegraphics[width=\textwidth]{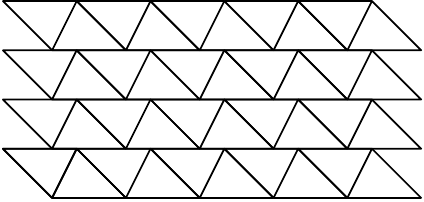}
\captionof{figure}{Tiling with one type of arbitrary triangle (D).}\label{pict:non-f2f-3tr-arbitrary}
\end{center}
\end{minipage}
\hfill
\begin{minipage}{0.45\textwidth}
\begin{center}
\includegraphics[width=\textwidth]{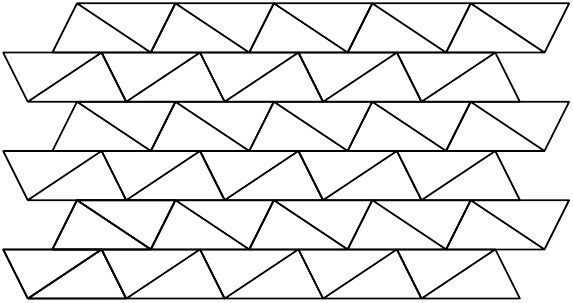}
\captionof{figure}{Tiling with one type of arbitrary triangle but 
odd rows are reflected.}\label{pict:non-f2f-tr-arbitrary-refl}
\end{center}
\end{minipage}

\begin{figure}[!ht]
\begin{center}
\includegraphics[width=0.55\textwidth]{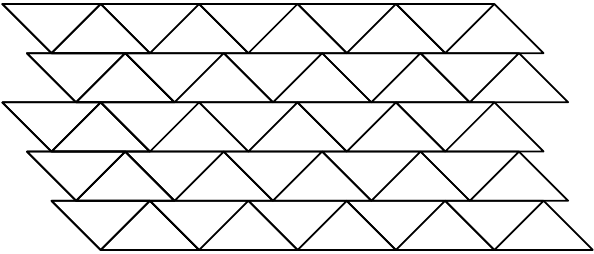}
\caption{Tiling with isosceles triangles (NC).}\label{pict:non-f2f-3tr-isosceles}
\end{center}
\end{figure}

\begin{minipage}{0.45\textwidth}
\begin{center}
\includegraphics[width=\textwidth]{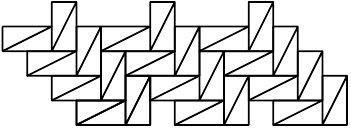}
\captionof{figure}{Tiling with one type of arbitrary triangle. 
Two triangles in the vertex corona share an entire long edge,
the other two share part of a medium 
edge.}\label{pict:non-f2f-3tr-arbitrary-twobig-onesmall}
\end{center}
\end{minipage}
\hfill
\begin{minipage}{0.45\textwidth}
\begin{center}
\includegraphics[width=\textwidth]{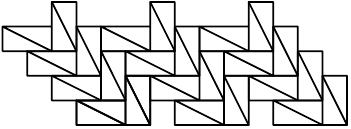}
\captionof{figure}{Tiling with one type of arbitrary triangle. 
Two triangles in the vertex corona share an entire long edge,
in the other two a medium edge touches a short 
edge.}\label{pict:non-f2f-tr-arbitrary-onebig-twosmall}
\end{center}
\end{minipage}

\begin{figure}[!ht]
\begin{center}
\includegraphics[width=0.6\textwidth]{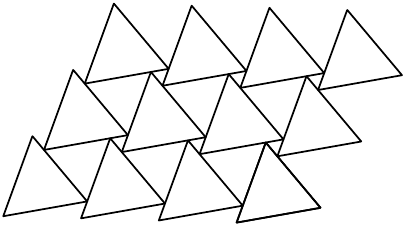}
\caption{Tiling with three types of regular triangles
(D).}\label{pict:non-f2f-3tr-three-regular}
\end{center}
\end{figure}

\section{Tilings with symmetry groups $\ast \ast$, 
$\times \times$, and $\circ$}\label{no-rotations}

\begin{figure}[!ht]
\includegraphics[width=0.8\textwidth]{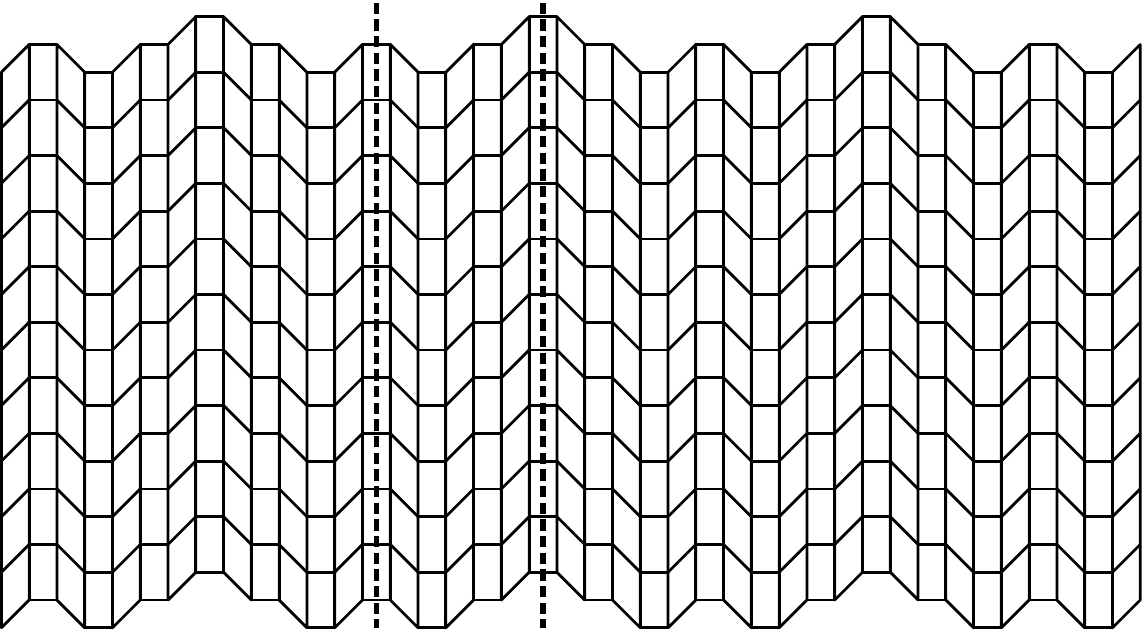}
\caption{A crystallographic tiling where layers repeat according to the 
periodic sequence $\ldots101100\ldots$. It is invariant under mirror 
reflection in the two axes indicated by dashed lines. \label{pict:star-star}}
\end{figure}

\begin{figure}[!ht]
\includegraphics[width=0.8\textwidth]{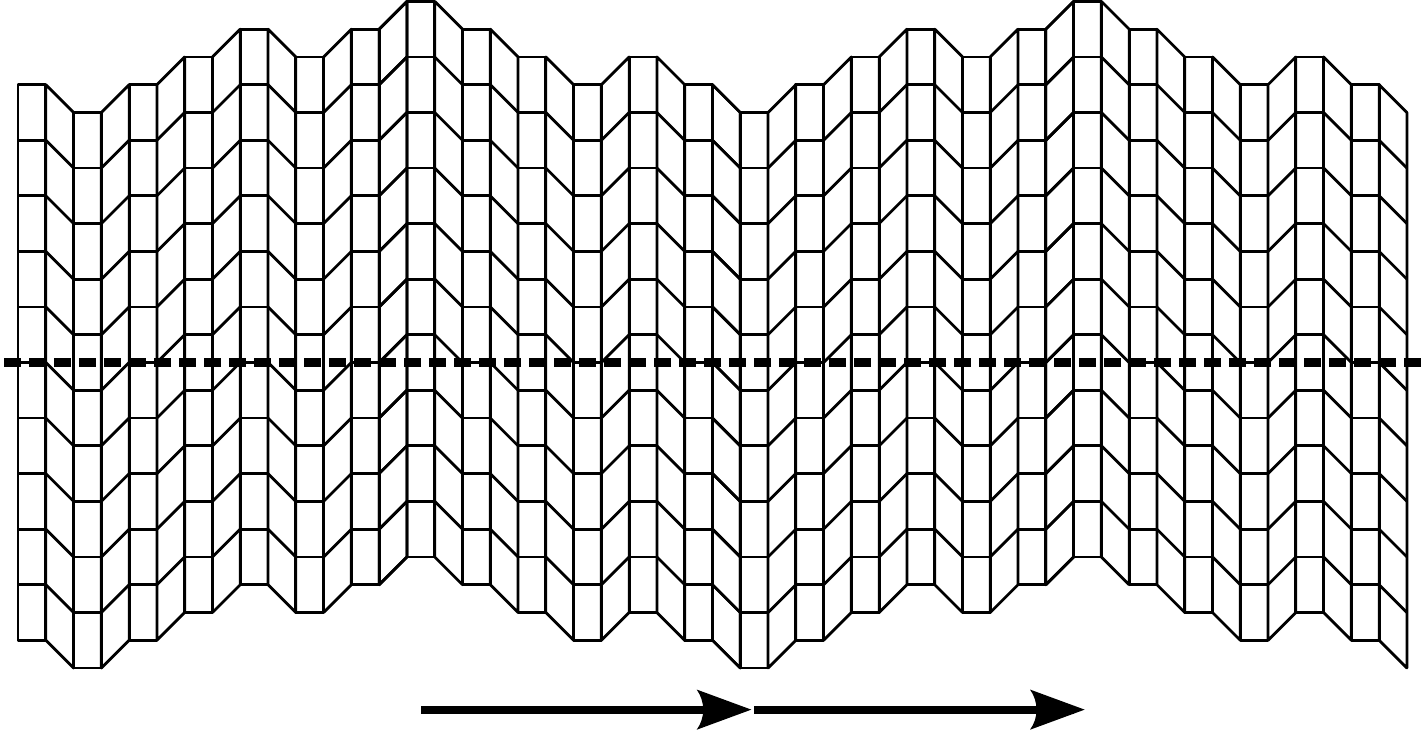}
\caption{A crystallographic tiling where layers repeat according to the 
periodic sequence $\ldots111011000100 \ldots$. It is not invariant under 
any reflection, but it is invariant under a glide reflection: a reflection in the 
axis indicated by the dashed line, followed by a translation 
indicated by one of the arrows.\label{pict:norefl-norot}}
\end{figure}

\begin{figure}[!ht]
\includegraphics[width=0.8\textwidth]{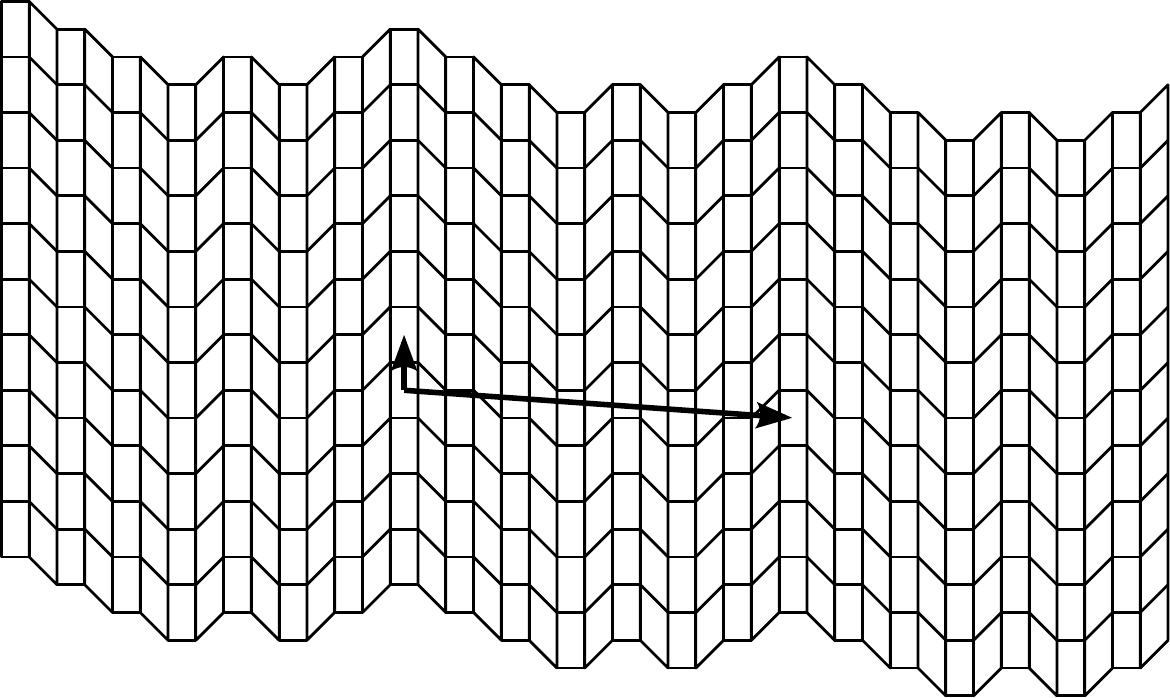}
\caption{A crystallographic tiling where layers repeat according to the 
periodic sequence $\ldots0001011 \ldots$. Its symmetry group contains 
translations only. The arrows in the image indicate two translation that 
generate the entire symmetry group.  \label{pict:only-transl}}
\end{figure}

\end{document}